\documentclass[11pt,leqno]{article}
\usepackage{amsmath}
\usepackage{subfigure}
\usepackage{amsthm}
\usepackage{amstext}
\usepackage{amsopn}
\usepackage{texdraw}
\usepackage{graphicx}
\usepackage{pstricks,pst-plot}
\usepackage{mathrsfs}

\usepackage{multicol,graphicx,color}
\usepackage{pslatex}
\usepackage{amsthm}
\usepackage{amsmath}
\usepackage{amssymb}
\usepackage{latexsym}
\usepackage{lscape}
\usepackage{epsfig}
\usepackage{pstricks}
\usepackage{amsfonts}
\usepackage{enumerate}
\usepackage{exscale}
\usepackage{relsize}

\newtheorem{theorem}{Theorem}[section]
\newtheorem{lemma}[theorem]{Lemma}
\newtheorem{proposition}[theorem]{Proposition}
\newtheorem{corollary}[theorem]{Corollary}
\theoremstyle{definition}

\theoremstyle{remark}
\newtheorem{remark}[theorem]{Remark}

\begin{document}
\title{The Broucke-H\'{e}non orbit and the Schubart orbit in the planar three-body problem with equal masses}
\author{
Wentian Kuang\\
Chern Institute of Mathematics, Nankai University\\
Tianjin 300071, China\\
Emails: kuangwt1234@163.com
\and
Tiancheng Ouyang \\
 Department of Mathematics, Brigham Young University\\
 Provo, Utah 84602, USA\\
 Email: ouyang@math.byu.edu
 \and
 Zhifu Xie\\
Department of Mathematics, University of Southern Mississippi\\
Hattiesburg, Mississippi 39406, USA\\
 Email: Zhifu.Xie@usm.edu
 \and
Duokui Yan \\
School of Mathematics and System Sciences, Beihang University\\
  Beijing 100191, China \\
  Email: duokuiyan@buaa.edu.cn
}
\date{}

\maketitle
\begin{abstract}
In this paper, we study the variational properties of two special orbits: the Schubart orbit and the Broucke-H\'{e}non orbit. We show that under an appropriate topological constraint, the action minimizer must be either the Schubart orbit or the Broucke-H\'{e}non orbit. One of the main challenges is to prove that the Schubart orbit coincides with the action minimizer connecting a collinear configuration with a binary collision and an isosceles configuration. A new geometric argument is introduced to overcome this challenge.  
\end{abstract}

{\bf AMS classification number:} 37N05, 70F10, 70F15, 37N30, 70H05, 70F17\\

\section{ Introduction}
H\'{e}non \cite{HM1, HM} numerically found a one-parameter family of periodic orbits in the planar equal-mass three-body problem, in which the angular momentum is chosen as the non-trivial parameter. In this family, there is a special orbit, as shown in Fig.\ref{Henonorbit}, which has a simple shape and good symmetry properties. This orbit was also independently discovered by Broucke \cite{BR} and it is called as the Broucke-H\'{e}non orbit in this paper. 

 \begin{figure}[ht]
    \begin{center}
    \subfigure[ \, Minimizing path of the Broucke-H\'{e}non orbit]{\includegraphics[width=2.7in, height=2.1in]{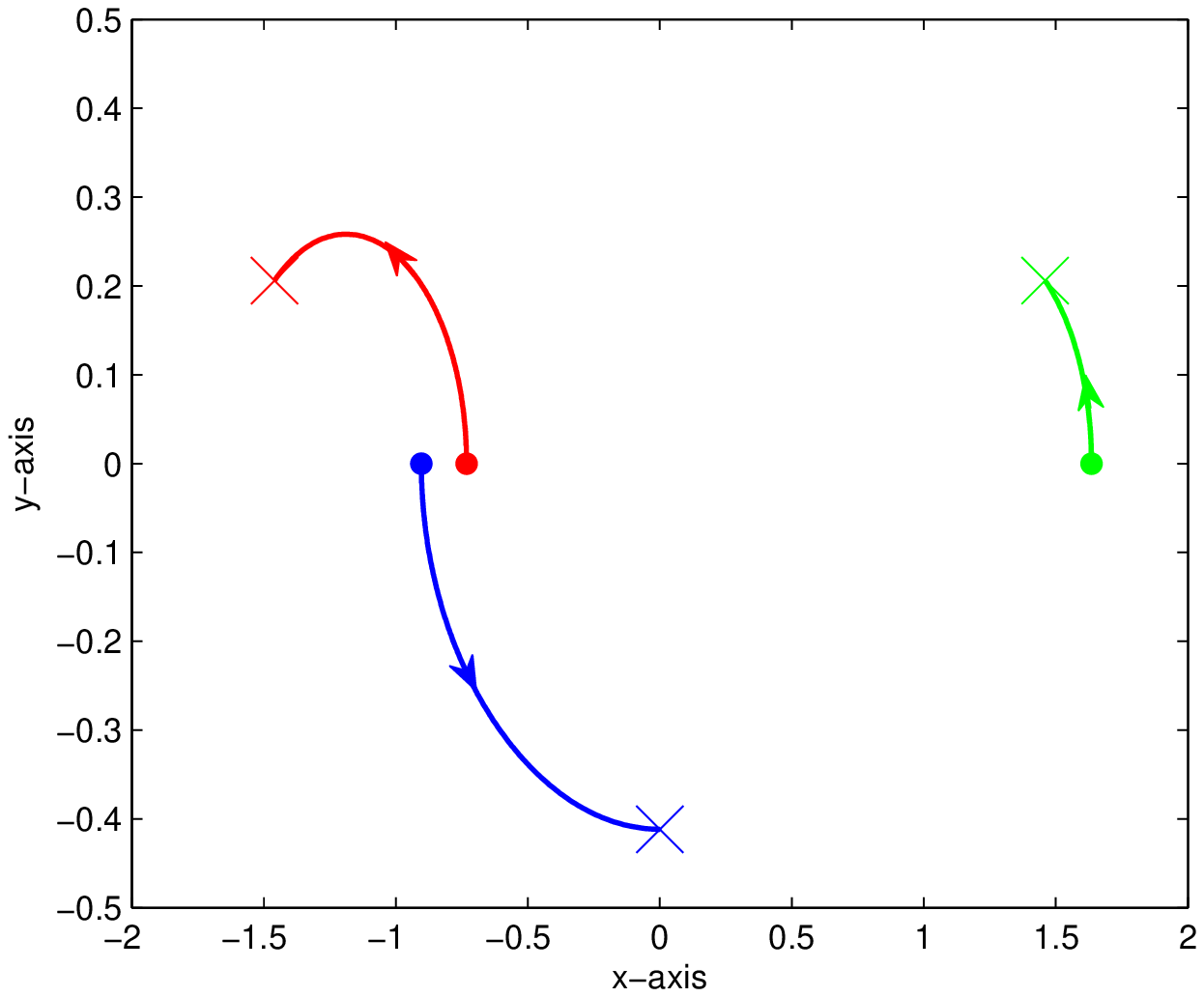}}
    \subfigure[ \, Trajectory of the Broucke-H\'{e}non orbit ]{\includegraphics[width=2.7in, height=2.1in]{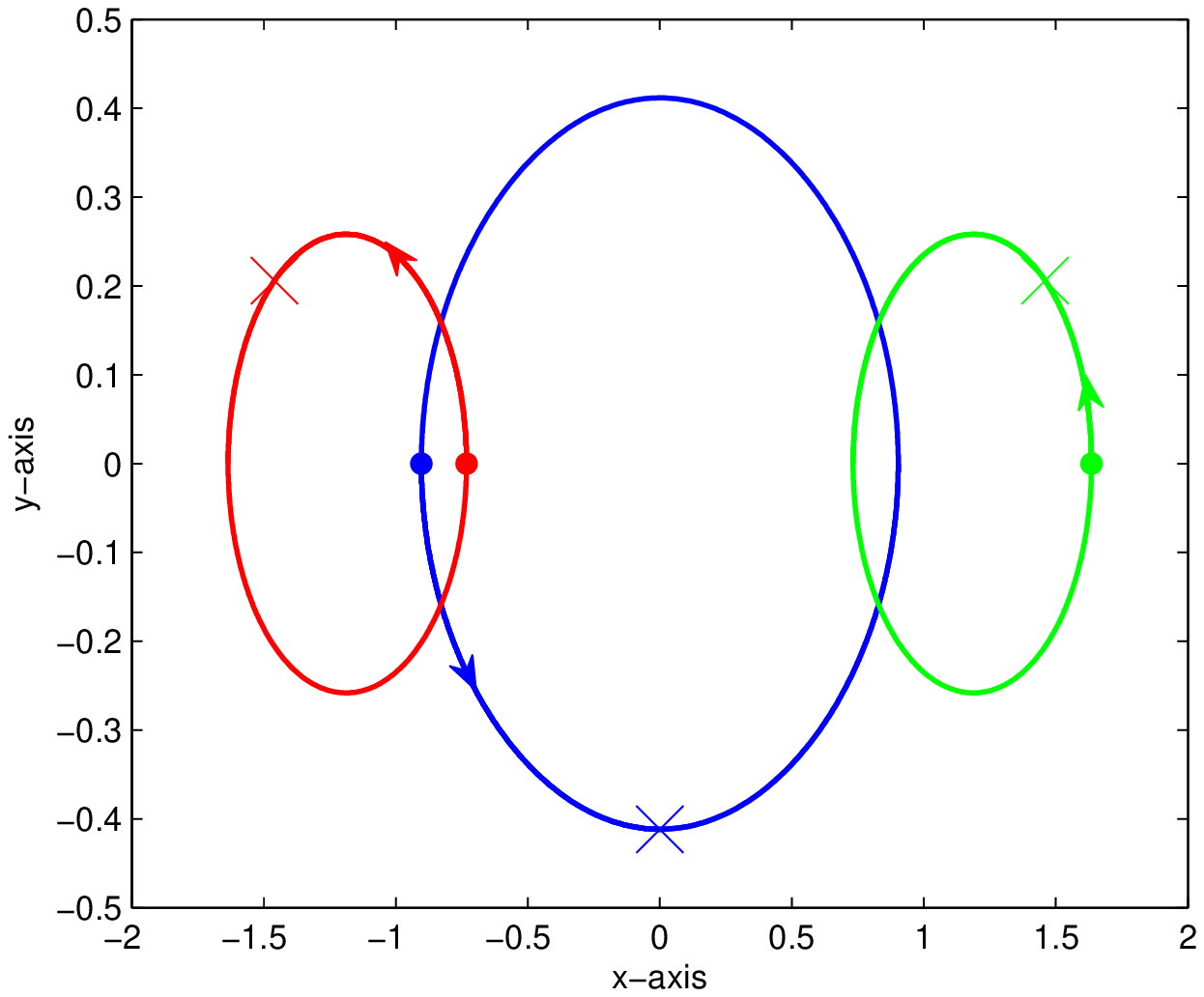}}
   \end{center}
 \caption{ \label{Henonorbit}Motion of the Broucke-H\'{e}non orbit. At $t=0$, the three masses (in dots) form a collinear configuration. At $t=1$, they (in crosses) form an isosceles configuration. The path shown in Fig. $($a$)$ is an action minimizer of a two-point free boundary value problem connecting a collinear configuration and an isosceles configuration.  Fig. $($b$)$ is the periodic solution extended by the path in Fig. $($a$)$.} 
 \end{figure}
In 2000, Chenciner and Montgomery \cite{CM} proved the existence of the figure-eight solution in the planar three-body problem with equal masses by using the variational method. Since then, a number of new periodic solutions have been discovered and proven to exist. A workshop on Variational Methods in Celestial Mechanics was organized by Chenciner and Montgomery in 2003 to address the possible applications of variational method in studying the Newtonian N-body problem, while several open problems are proposed by the attending experts. The existence of the Broucke-H\'{e}non orbit is one of these open problems, which was proposed by A. Venturelli. Actually, he noticed that the Schubart orbit with collision \cite{Moe, Sch}  (Fig. \ref{Schubartorbit}) is on the closure of the homology class $(1,\, 0, \, 1)$ and it is not clear if the Broucke-H\'{e}non orbit (Fig. \ref{Henonorbit}) is a minimizer in the homology class $(1, \, 0,\, 1)$.
\begin{figure}[ht]
    \begin{center}
\subfigure[ \, Minimizing path of the Schubart orbit]{\includegraphics[width=2.7in, height=2.1in]{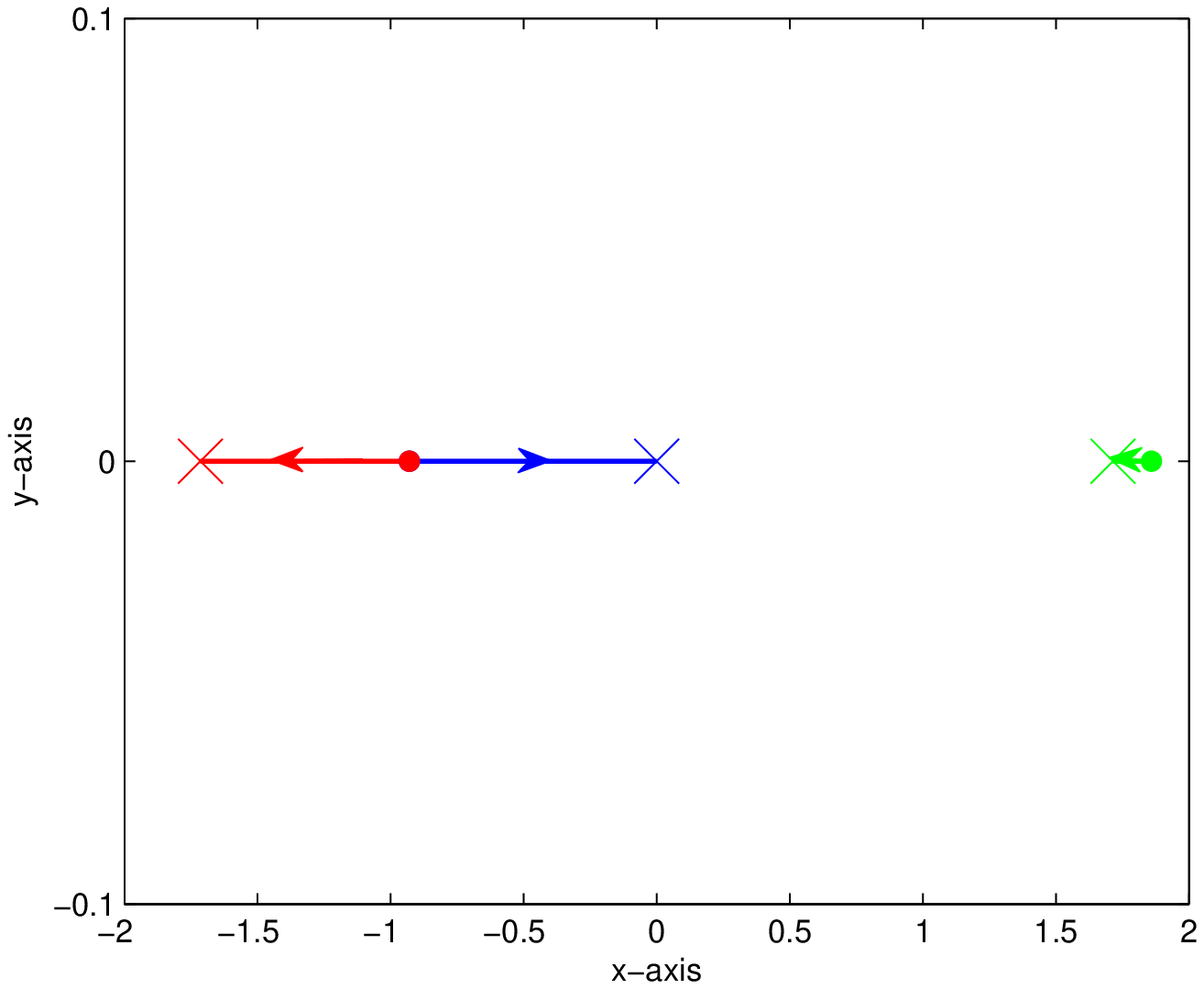}}
\subfigure[ \, Trajectory of the Schubart orbit]{\includegraphics[width=2.7in, height=2.1in]{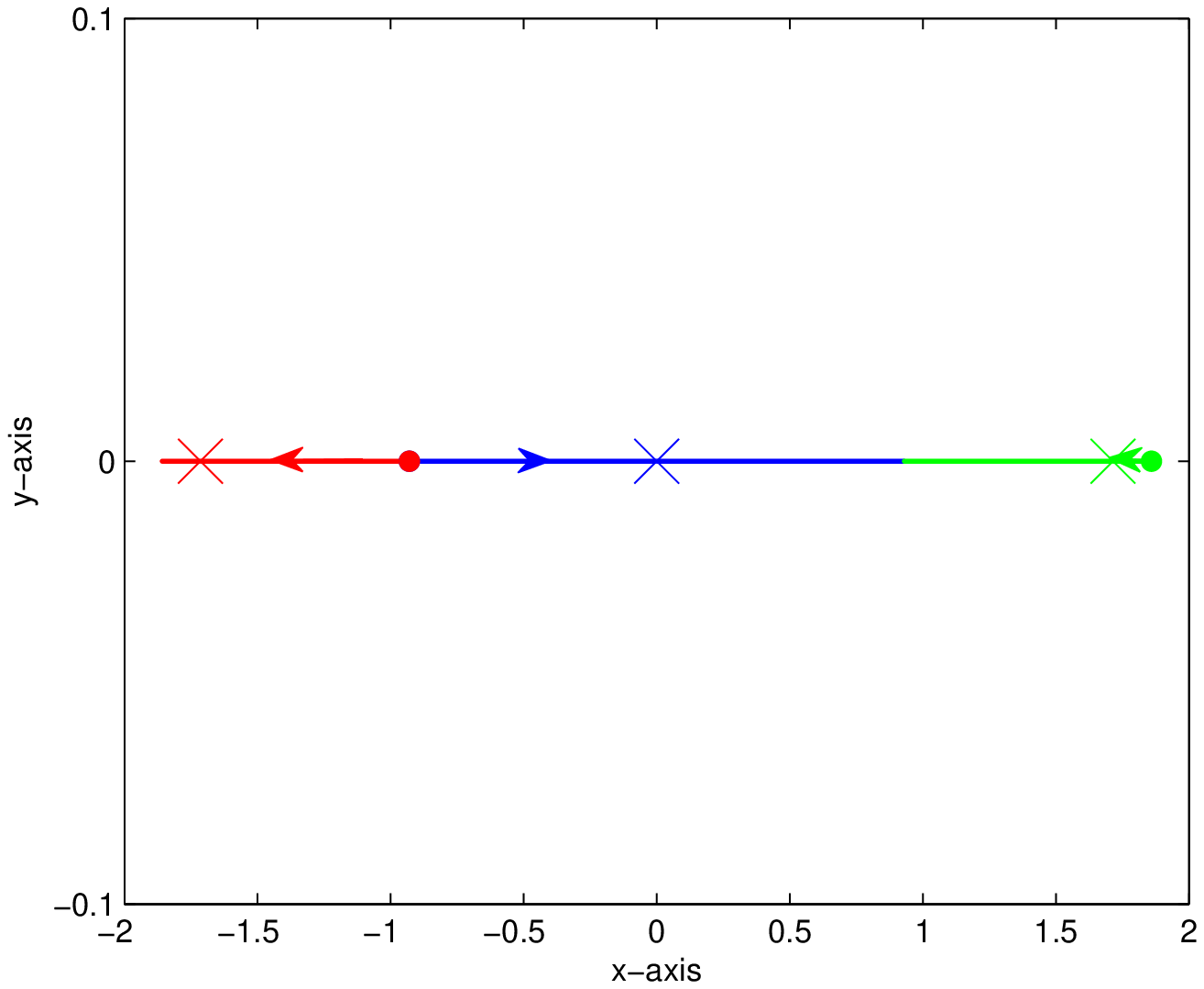}} 
 \end{center}
\caption{ \label{Schubartorbit}Motion of the Schubart orbit. At $t=0$, body 1(blue dot) and body 2 (red dot) collide on the $x-$axis, while body 3 (green dot) stays away from them. At $t=1$, they (in crosses) form an Euler configuration. The path in Fig. $($a$)$ is one fourth of the Schubart orbit, which is an action minimizer in a two-point free boundary value problem connecting a collinear configuration with a binary collision and an Euler configuration.  Fig. $($b$)$ is the Schubart orbit extended from the path in Fig. $($a$)$.}
 \end{figure}
 
In this paper, we use a variational approach to study the existence of the Broucke-H\'{e}non orbit and the Schubert orbit. The variational method we used is based on a two-point free boundary value problem. By using this approach, many stable choreographic solutions \cite{OuXie} and other periodic solutions \cite{YAN2, YAN4, YAN5} have been found numerically and proved theoretically. In order to obtain the Broucke-H\'{e}non orbit and the Schubart orbit, we first choose appropriate structural prescribed boundary conditions in the variational frame.

Let $q_i=q_i(t) \, (i=1, 2, \dots, N)$ denote the position of mass $m_i$ in $\mathbb{R}^d$. Set $q=\begin{bmatrix}
q_1 \\
q_2 \\
\dots \\
q_N
\end{bmatrix}$ to be an $N \times d$ matrix, where $N$ is the number of bodies and $d$ is the dimension.  Without loss of generality, we assume that the center of mass is always at the origin. Let \[\chi= \left\{q\, \bigg|  \, \sum_{i=1}^{N} m_i q_i =0\right\}. \] The Lagrangian action functional $\mathcal{A}$ is defined as follows
\begin{equation}\label{Action}
\mathcal{A}=\mathcal{A}(q(t), \dot{q}(t))=\int_{0}^{1} (K+U) \, dt,
\end{equation}
where $\displaystyle K=K(\dot{q}(t))=\frac{1}{2}\sum_{i=1}^{N} m_i|\dot{q}_i(t)|^2$ is the kinetic energy and  $\displaystyle U=U(q(t))=\sum_{1\leq i<j\leq N}\frac{m_i m_j}{|{q}_i(t)-{q}_j(t)|}$ is the Newtonian
potential. It is known that collision-free critical points of the action functional $\mathcal{A}$ are trajectories satisfying the equations of motion, i.e. the Newtonian equations:
\begin{equation}\label{Newton}
m_i\ddot{q}_i=\frac{\partial U}{\partial q_i}=\sum_{j=1,j\not= i}^{N} \frac{m_im_j(q_j-q_i)}{|q_j-q_i|^3}, \,  \hspace{1cm} 1\leq i\leq N.
\end{equation}

Instead of studying a periodic solution directly, a segment of a periodic solution will be considered in a two-point free boundary value problem. The variational method we use is then a two-step minimizing procedure. First, we consider a fixed boundary value problem, which is also known as the Bolza problem. For given boundary matrices $q(0)$ and $q(1)$, there exists an action minimizer $\mathcal{P}$ connecting them.  By Marchal \cite{Ma1} and Chenciner's \cite{CA2} work, this minimizer $\mathcal{P}$ is collision free except the possible collisions at the boundary points. However, if one wants $\mathcal{P}$ to be a part of a periodic solution, the two boundaries must be quite special. Hence, we introduce a second minimizing procedure. Instead of fixing the boundaries, we free several parameters on the boundaries $q(0)$ and $q(1)$. The Lagrangian action functional is then minimized over these parameters. The resulting minimizing path may be extended to a periodic solution or a quasi-periodic solution. There are mainly three challenges to show the existence of such classical solutions in the variational method. The first is the existence of minimizer of the functional under the boundary constraints. The second is the collision-free of the minimizer on the boundaries. The third is whether the minimizing path can be extended to a periodic solution that we have expected. With appropriate choices of the free boundaries, the three challenges can be resolved.  

To introduce this variational method in detail, we define two boundary matrices $Q_s$ and $Q_e$ as follows:
\begin{equation}\label{Q_sQ_e1}
Q_s= \begin{bmatrix}
q_1(a_1, \dots a_k) \\
\dots \\
q_N(a_1, \dots a_k) 
\end{bmatrix}, \qquad  Q_e= \begin{bmatrix}
q_1(b_1, \dots b_s) \\
\dots \\
q_N(b_1, \dots b_s) 
\end{bmatrix}, 
\end{equation}
where $q_i \in \mathbb{R}^d \, (d=1,2,3; \, i=1, \dots, N)$ and $\vec{\alpha}=(a_1, a_2, \dots, a_k), \, \vec{\beta}=(b_1, b_2, \dots, b_s)$ are independent variables. Let 
\[P(Q_s, \, Q_e)= \left\{ q(t) \in H^1([0, 1], \chi) \,| \, q(0)=Q_s, \, q(1)=Q_e \right\}.\]
The following problem is considered:
\begin{equation}\label{introtwostepminprob}
\inf_{ (\vec{\alpha}, \, \vec{\beta}) \in  \mathcal{S}  }  \quad \inf_{ q(t) \in P(Q_s, Q_e) }  \mathcal{A},
\end{equation}
where $\mathcal{S} $ is  a closed subset in $R^{k+s}$ and $\mathcal{A}= \int_0^1 (K+U) \, dt$. 
 The first question is the coercivity of the functional $\mathcal{A}$ in the minimizing problem \eqref{introtwostepminprob}. By making some general assumptions on $Q_s$ and $Q_e$, we show that the coercivity of the functional $\mathcal{A}$ in \eqref{introtwostepminprob} holds.
\begin{theorem}\label{Thm1.1}       
Let 
\begin{equation}\label{Q_sQ_e}
q(0)=Q_s= \begin{bmatrix}
q_1(a_1, \dots, a_k) \\
\dots \\
q_N(a_1, \dots, a_k) 
\end{bmatrix}, \qquad  q(1)=Q_e= \begin{bmatrix}
q_1(b_1, \dots, b_s) \\
\dots \\
q_N(b_1, \dots, b_s) 
\end{bmatrix}, 
\end{equation}
where $Q_s, Q_e \in \chi$, $q_i \in \mathbb{R}^d (i=1, \dots, N)$ and $\vec{\alpha}=(a_1, \, \dots, \, a_k), \,\vec{\beta}=( b_1, \,  \dots,  \, b_s)$ are independent variables. Let $\vec{\alpha} \in \mathcal{S}_1$,  $\vec{\beta} \in  \mathcal{S}_2$, where $ \mathcal{S}_1 \subset \mathbb{R}^k$ and $ \mathcal{S}_2 \subset \mathbb{R}^s$ are closed subsets.  $\mathcal{S}_1$ and $\mathcal{S}_2$ have either finitely many lines as its boundary or no boundary. Let $\mathcal{S}_1 \cup\mathcal{S}_2=\mathcal{S}$.  Assume that both $\left\{Q_s \, \big| \, \vec{\alpha}  \in  \mathbb{R}^k \right\}$ and $ \{Q_e  \, \big| \,  \vec{\beta}  \in  \mathbb{R}^s  \}$ are linear subspaces. If the intersection of the two configuration subsets is at origin or equal to an empty set, i.e. 
\[ \left\{Q_s \, \big|\,  \vec{\alpha}  \in  \mathcal{S}_1 \right\} \cap  \{Q_e  \, \big| \, \vec{\beta} \in  \mathcal{S}_2 \}=\{\vec{0} \} \, \, \text{or} \,  \, \varnothing, \]
then there exist a path sequence $\{ \mathcal{P}_{n_l} \}$ and a minimizer $\mathcal{P}_0$ in $H^1([0, 1], \chi)$, such that for each $n_l$,
\[ \mathcal{A}(\mathcal{P}_{n_l})= \inf_{q \in P(Q_{s_{n_l}},Q_{e_{n_l}}) }  \mathcal{A},\]
\begin{equation*}
\mathcal{A}(\mathcal{P}_0) = \inf_{  (\vec{\alpha}, \vec{\beta}) \in  \mathcal{S}  } \quad \inf_{ q \in P(Q_s,Q_e) }  \mathcal{A}= \inf_{q \in P(Q_{s_{0}},Q_{e_{0}}) }  \mathcal{A},
\end{equation*}
where \[P(Q_{s_{n_l}},Q_{e_{n_l}})= \left\{q \in H^1([0,1],\chi) \,\big{|}  \,q(0)=Q_s(\vec{\alpha}_{n_l}), \,  q(1)=Q_e(\vec{\beta}_{n_l}) \right\}\] and 
\[P(Q_{s_{0}},Q_{e_{0}}) =  \left\{q \in H^1([0,1],\chi) \,\big{|}  \,q(0)=Q_s(\vec{\alpha}_{0}), \, q(1)= Q_e(\vec{\beta}_{0})  \right\}\] with $\vec{\alpha}_{n_l}= (a_{{1}_{n_l}},\dots, a_{{k}_{n_l}})$ and $\vec{\beta}_{n_l}= (b_{{1}_{n_l}},\dots, b_{{s}_{n_l}})$.
 For $t \in [0, 1]$,  $\displaystyle  \mathcal{P}_{n_l}(t)$ converges to $\mathcal{P}_0 (t)$ uniformly. In particular,
\[ \lim_{n_l \to \infty} \vec{\alpha}_{{n_l}} = \vec{\alpha}_{{0}}, \qquad  \lim_{n_l \to \infty} \vec{\beta}_{{n_l}} = \vec{\beta}_{{0}}.\]
\end{theorem}

As an application of Theorem \ref{Thm1.1}, there exists an action minimizer $\mathcal{P}$, such that
\begin{equation}\label{thispapergoal}
\mathcal{A}(\mathcal{P})= \inf_{ \{ (a_1,\, a_2, \, b_1, \, b_2) \in  \mathcal{S}   \} }  \quad \inf_{\{ q(0)= Q_{s_1}, \, q(1)= Q_{e_1}, \, q(t) \in H^1([0, 1], \chi) \} }  \mathcal{A},
\end{equation}
where $m_1=m_2=m_3=1$, $\mathcal{A}=\int_{0}^{1} (K+U) \, dt$ is defined by \eqref{Action}, 
\begin{equation}\label{Henon1}
 Q_{s_1} = \begin{bmatrix}
-2a_1-a_2   &  0   \\
a_1-a_2  &  0     \\
a_1+2a_2   & 0   
\end{bmatrix}, \qquad \,\,
  Q_{e_1} =  \begin{bmatrix}
0   &  -2b_1    \\
-b_2 &  b_1 \\
b_2   &   b_1
\end{bmatrix},  
\end{equation}
and  $(a_1, \, a_2, \, b_1, \, b_2) \in  \mathcal{S}$, with 
\[\mathcal{S} = \left\{  a_1 \geq 0, \, \,  a_2 \geq 0, \, \,   b_1 \in \mathbb{R}, \, \,   b_2 \in \mathbb{R}  \right\}.  \]

Based on the celebrated works of Chenciner \cite{CA2} and Marchal \cite{Ma1}, the minimizer $\mathcal{P}$ is collision-free for $t \in (0,1)$. The challenge is to exclude possible boundary collisions in $\mathcal{P}$. A geometric argument is introduced to study the action minimizer $\mathcal{P}$ under the Jacobi coordinate system in Section \ref{5.2}. This geometric argument was first introduced in \cite{YK} to study the retrograde orbits and the prograde orbits in the parallelogram equal-mass four-body problem, which can be applied to show new geometric properties of the orbits in the planar three-body and four-body problem. Our main result is as follows, while its proof can be found in Theorem \ref{mincolschubart} and Lemma \ref{henonorbit1}. 
\begin{theorem}
The minimizer $\mathcal{P}$ in \eqref{thispapergoal} is either one part of the Schubart orbit or one part of the Broucke-H\'{e}non orbit.
\end{theorem}

\begin{remark}
Numerically, we can calculate the action value of $\mathcal{P}$ in both orbits: the Schubart orbit and the Broucke-H\'{e}non orbit. If $\mathcal{P}$ is one part of the Schubart orbit, its action value $\mathcal{A}_{10}$ is
\[\mathcal{A}_{10}  \approx 3.43.  \]
If $\mathcal{P}$ is one part of the Broucke-H\'{e}non orbit, its action value $\mathcal{A}_{20}$ is 
\[\mathcal{A}_{20} \approx  3.46. \]
Therefore, numerical evidence implies that the action minimizer $\mathcal{P}$ of \eqref{thispapergoal} actually coincides with the Schubart orbit. 
\end{remark} 

The paper is organized as follows. Section \ref{coercivityresult} introduces a general result of coercivity. Section \ref{exclusiontotalcol} excludes the total collision in the minimizer $\mathcal{P}$ of \eqref{thispapergoal}. Section \ref{revisitkepler} excludes possible binary collisions of $\mathcal{P}$ at $t=1$ and studies the behavior of binary collisions in $\mathcal{P}$ at $t=0$. Section \ref{binarycollisionat0} shows that $\mathcal{P}$ must coincide with the Schubart orbit in Fig. \ref{Schubartorbit} if $\mathcal{P}$ has collision singularities.  In the last section, we show that if $\mathcal{P}$ has no collision, it can be extended to a periodic orbit with $D_2$ symmetry.

\section{Coercivity under general boundary settings}\label{coercivityresult}
In this section, we prove Theorem \ref{Thm1.1} of the coercivity of the Lagrangian action functional in the N-body problem. Actually, similar coercivity results can be found in \cite{Chen2, Chen3, Fusco}.

Note that $L=K+U \geq 0$, hence there exists some $M_0 \geq 0$, such that 
\[\inf_{ (\vec{\alpha}, \, \vec{\beta}) \in  \mathcal{S}  }  \quad \inf_{ q(t) \in P(Q_s, Q_e) }   \mathcal{A} =M_0,\]
where $\vec{\alpha}=(a_1, \dots, a_k)$ and $\vec{\beta}=(b_1, \dots, b_s)$. The proof follows by the Arzel\`{a}-Ascoli theorem. Basically, we can find a sequence $\{\mathcal{P}_{n}\}$, such that the action of the sequence $\mathcal{A}(\mathcal{P}_{n})$ approaches $M_0$. Then we show the uniform boundedness and equicontinuity of the sequence. Hence, by the Arzel\`{a}-Ascoli theorem, there is a subsequence $ \{ \mathcal{P}_{n_l} \}$ which converges uniformly to a minimizer $\mathcal{P}_0$.  

Note that there exist sequences $\{a_{{i}_{n}}\}$ and $\{b_{{j}_{n}}\}$, such that the minimum action value $M_0$ can be reached by a path sequence $\mathcal{P}_{n} \in H^1([0, 1], \chi)$, which satisfies 
\[ \mathcal{A}(\mathcal{P}_n)= \inf_{\{ q(0)= Q_s, \, q(1)= Q_e, \, q(t) \in H^1([0, 1], \chi), \, a_i=a_{{i}_{n}}, \, b_j=b_{{j}_{n}}, (i=1, \dots, k; j=1, \dots, s) \} }  \mathcal{A}, \]
and $ \mathcal{A}(\mathcal{P}_n) \in [ M_0, \, M_0+ \frac{1}{2n}]$. It is clear that $\mathcal{A}(\mathcal{P}_n) \in [ M_0, \, M_0+1]$ for all $n$. Next, we show the path sequence $\{\mathcal{P}_n\}$ is uniformly bounded. 

We rewrite $Q_s$ and $Q_e$ as $dN \times 1$ vectors:
\[\widetilde{Q_s}= \begin{bmatrix}
q_1^{T}(\vec{\alpha}) \\
\dots \\
q_{N}^{T}(\vec{\alpha}) 
\end{bmatrix}, \qquad  \widetilde{Q_e} = \begin{bmatrix}
q_1^{T}(\vec{\beta}) \\
\dots \\
q_{N}^{T}(\vec{\beta}) 
\end{bmatrix}.\]
By assumptions, it follows that both $\left\{Q_s \, \big| \, \vec{\alpha}  \in  \mathcal{S}_1 \right\} $ and $\left\{Q_e \,  \big| \, \vec{\beta}  \in  \mathcal{S}_2 \right\} $ are closed and they have either finitely many lines as its boundary or no boundary. Furthermore, they satisfy
\[ \left\{Q_s \, \big|\,  \vec{\alpha}  \in  \mathcal{S}_1 \right\}  \cap\left\{Q_e \, \big| \, \vec{\beta}  \in  \mathcal{S}_2 \right\}= \{\vec{0} \} \, \, \text{or} \,  \, \varnothing. \]
Note that $\left\{\widetilde{Q_s} \, \big|\,  \vec{\alpha} \in  \mathcal{S}_1 \right\}$ is a closed subset of a $k$-dimensional linear space $U_k= \left\{\widetilde{Q_s} \, \big| \, \vec{\alpha} \in  \mathbb{R}^k \right\}$ and  $\left\{\widetilde{Q_e}  \, \big| \, \vec{\beta} \in  \mathcal{S}_2 \right\}$ is a closed subset of a $s$-dimensional linear space $V_s=\left\{\widetilde{Q_e}  \, \big|\,  \vec{\beta} \in \mathbb{R}^s  \right\}$. Let $\{u_1, \dots, u_k\}$ be an orthonormal basis of $U_k$ and $\{v_1, \dots, v_s\}$ be an orthonormal basis of $V_s$. 

If both $\left\{\widetilde{Q_s} \, \big| \,  \vec{\alpha} \in  \mathcal{S}_1 \right\}$ and  $\left\{\widetilde{Q_e}  \, \big|\,  \vec{\beta}  \in  \mathcal{S}_2 \right\}$ are unbounded, we can choose nonzero vectors \[\vec{u} \in  \left\{\widetilde{Q_s} \, \big|\,  \vec{\alpha} \in  \mathcal{S}_1 \right\}, \quad \vec{v} \in  \left\{\widetilde{Q_e} \, \big|\,  \vec{\beta} \in  \mathcal{S}_2 \right\},\] such that $k\vec{u} \in  \left\{\widetilde{Q_s} \, \big|\,  \vec{\alpha} \in  \mathcal{S}_1 \right\}$ and $k \vec{v} \in  \left\{\widetilde{Q_e} \, \big| \, \vec{\beta} \in  \mathcal{S}_2 \right\}$ for any $k>k_1$ with $k_1$ big enough. For the vectors $\vec{u}$ and $\vec{v}$, there exist constants $g_i, \, h_j \, (1\leq i \leq k, \, 1 \leq j \leq s)$, such that
\[ \frac{\vec{u}}{|\vec{u}|}= g_1 u_1+ \dots + g_k u_k, \qquad \sum_{i=1}^k g_i^2=1, \]
\[  \frac{\vec{v}}{|\vec{v}|}= h_1 v_1+ \dots + h_s v_s, \qquad \sum_{j=1}^s h_j^2=1. \]
Note that $g_i$ and $h_j \, (1\leq i \leq k, \, 1 \leq j \leq s)$ satisfy $\displaystyle \sum_{i=1}^k g_i^2=\sum_{j=1}^s h_j^2=1$. So they are on a compact set. Since both $ \left\{\widetilde{Q_s} \, \big|\,  \vec{\alpha} \in  \mathcal{S}_1 \right\}$ and $\left\{\widetilde{Q_e} \, \big| \,  \vec{\beta} \in  \mathcal{S}_2 \right\}$ are closed, it follows that the inner product of $\frac{\vec{u}}{|\vec{u}|}$ and $\frac{\vec{v}}{|\vec{v}|}$:
\[ <\frac{\vec{u}}{|\vec{u}|}, \, \, \frac{\vec{v}}{|\vec{v}|}> = \sum_{1\leq i \leq k, \, 1 \leq j \leq s} g_i h_j <u_i, v_j>= \cos (\vec{u}, \vec{v}) \]
can reach its maximum $K_0$. By the assumptions on the boundaries of $\mathcal{S}_i \, (i=1,2)$, it follows that $K_0$ is determined only by the sets $\mathcal{S}_1$ and $\mathcal{S}_2$. If $K_0=1$, there exist two vectors $\vec{u} \in \left\{\widetilde{Q_s} \, \big| \, \vec{\alpha} \in  \mathcal{S}_1 \right\}$ and $\vec{v} \in \left\{\widetilde{Q_e} \, \big|\,  \vec{\beta} \in  \mathcal{S}_2 \right\}$, such that $\displaystyle \frac{\vec{u}}{|\vec{u}|}= \frac{\vec{v}}{|\vec{v}|}$.  By assumption, it follows that there exists big enough $k_2>0$, such that
\[ k_2 \vec{u} \in \left\{\widetilde{Q_s} \, \big|\,  \vec{\alpha} \in  \mathcal{S}_1 \right\} \cap \left\{\widetilde{Q_e} \, \big| \,  \vec{\beta} \in  \mathcal{S}_2 \right\}.  \]
Contradiction! Hence, $K_0 <1$. It implies that for any two unbounded directions \[\vec{u} \in  \left\{\widetilde{Q_s} \, \big| \,  \vec{\alpha} \in  \mathcal{S}_1 \right\}, \quad \vec{v} \in  \left\{\widetilde{Q_e} \, \big| \,  \vec{\beta} \in  \mathcal{S}_2 \right\},\] 
the angle between them is strictly less than $\pi$.

On the other hand, $\mathcal{A}(\mathcal{P}_n)\leq M_0+1$. If $0 \leq t_1 < t_2 \leq 1$, we have
\begin{equation*}
\frac{m_j |q_{j}^{n}(t_2)- q_{j}^{n}(t_1)|^2}{2d(t_2-t_1)} \leq \int_{t_1}^{t_2} \frac{m_j |\dot{q}_j(t)|^2}{2} \, dt \leq  \mathcal{A}(\mathcal{P}_n)\leq M_0+1. \end{equation*}
It impies that for any $1\leq j \leq N$ and any $t_1$, $t_2$ satisfying $0 \leq t_1 < t_2 \leq 1$,
\begin{equation}\label{ineqalqq}
|q_j^{n}(t_2)- q_j^{n}(t_1)| \leq \sqrt{\frac{2d(t_2-t_1)(M_0+1)}{m_j}}.\end{equation}
Let $m^{*}=min\{m_1, m_2, \cdots, m_N\}$. Then for all $1 \leq j \leq N$, and any $t_1$, $t_2$ satisfying $0 \leq t_1 < t_2 \leq 1$, 
\[ |q_j^{n}(t_2)- q_j^{n}(t_1)| \leq \sqrt{\frac{2d(M_0+1)}{m^{*}}}. \]
In each $\mathcal{P}_{n}$, its element $q^{(n)}(t)=\begin{bmatrix}
q_1^{(n)}(\vec{\alpha}) \\
\dots \\
q_N^{(n)}(\vec{\alpha}) 
\end{bmatrix}$ can be rewritten as  $\tilde{q}^{(n)}(t)= \begin{bmatrix}
q_{1}^{(n)}(\vec{\alpha})^{T} \\
\dots \\
q_{N}^{(n)}(\vec{\alpha})^{T} 
\end{bmatrix}$ . Then for any $t \in [0, 1]$, 
\begin{equation}\label{boundq}
 |\tilde{q}^{(n)}(0)-\tilde{q}^{(n)}(t)| \leq  N\sqrt{\frac{2d(M_0+1)}{m^{*}}}. 
 \end{equation}
 The uniform boundedness is discussed in two cases. If both  $\frac{\tilde{q}^{(n)}(0)}{|\tilde{q}^{(n)}(0)|}$ and $\frac{\tilde{q}^{(n)}(1)}{|\tilde{q}^{(n)}(1)|}$ are unbounded directions in the corresponding boundary configuration sets, the assumption implies that the two vectors satisfy 
 \[k \frac{\tilde{q}^{(n)}(0)}{|\tilde{q}^{(n)}(0)|} \in  \left\{\widetilde{Q_s} \, \big| \, \vec{\alpha} \in  \mathcal{S}_1 \right\}, \quad k\frac{\tilde{q}^{(n)}(1)}{|\tilde{q}^{(n)}(1)|} \in  \left\{\widetilde{Q_e} \, \big| \,  \vec{\beta} \in  \mathcal{S}_2 \right\}\] for all $k>k_0$ with $k_0$ big enough. It follows that
\begin{align*}
\frac{2dN^2(M_0+1)}{m^{*}}  & \geq  |\tilde{q}^{(n)}(0)-\tilde{q}^{(n)}(1)|^2\\
& = |\tilde{q}^{(n)}(0)|^2+|\tilde{q}^{(n)}(1)|^2 -2 |\tilde{q}^{(n)}(0)| |\tilde{q}^{(n)}(1)| \cos \left(\tilde{q}^{(n)}(0), \, \tilde{q}^{(n)}(1) \right)\\
&\geq  |\tilde{q}^{(n)}(0)|^2+|\tilde{q}^{(n)}(1)|^2 -2 K_0 |\tilde{q}^{(n)}(0)||\tilde{q}^{(n)}(1)| \\
&= \left[ K_0 |\tilde{q}^{(n)}(0)|  -  |\tilde{q}^{(n)}(1)|\right]^2 + (1-K_0^2) |\tilde{q}^{(n)}(0)|^2\\
&\geq  (1-K_0^2) |\tilde{q}^{(n)}(0)|^2.
\end{align*}
Hence
\begin{equation}\label{boundq0}
|\tilde{q}^{(n)}(0)|  \leq \sqrt{\frac{2dN^2(M_0+1)}{m^{*}(1-K_0^2)}}. 
\end{equation}
By inequalities \eqref{boundq} and \eqref{boundq0}, it follows that for any $t \in [0, 1]$,
\begin{equation}\label{unifbound}
|\tilde{q}^{(n)}(t)|  \leq  |\tilde{q}^{(n)}(0)-\tilde{q}^{(n)}(t)| + |\tilde{q}^{(n)}(0)|  \leq N\sqrt{\frac{2d(M_0+1)}{m^{*}}}+ N \sqrt{\frac{2d(M_0+1)}{m^{*}(1-K_0^2)}},
\end{equation}
which is a uniform bound for $|\tilde{q}^{(n)}(t)| $. 

The other case is that one of the directions $\frac{\tilde{q}^{(n)}(0)}{|\tilde{q}^{(n)}(0)|}$ or $\frac{\tilde{q}^{(n)}(1)}{|\tilde{q}^{(n)}(1)|}$ stays bounded in its corresponding configuration subset. Without loss of generality, we assume that $\frac{\tilde{q}^{(n)}(0)}{|\tilde{q}^{(n)}(0)|}$ is a bounded direction in $ \left\{\widetilde{Q_s} \, \big| \, \vec{\alpha} \in  \mathcal{S}_1 \right\}$ as $n \to +\infty$. By assumptions on the boundary sets, there exists some $K_1>0$ determined only by the set $\left\{\widetilde{Q_s} \, \big| \, \vec{\alpha} \in  \mathcal{S}_1 \right\}$, such that $|\tilde{q}^{(n)}(0)|\leq K_1$. By inequality \eqref{boundq}, $\{\tilde{q}^{(n)}(t)\}$ is uniformly bounded. Therefore, the path sequence $\{\mathcal{P}_n= \mathcal{P}_n(t)\}$ is uniformly bounded.

Next, we show  the path sequence $\{ \mathcal{P}_n= \mathcal{P}_n(t) \}$ is equi-continuous. In fact, by inequality \eqref{ineqalqq},
\[ |q_j^{n}(t_2)- q_j^{n}(t_1)| \leq \sqrt{\frac{2d(M_0+1)}{m^{*}}} |t_2-t_1|^{1/2}.\]
  Then for any $\varepsilon>0$,  let $\delta= \frac{\varepsilon^2 m^{*}}{6(M_0+1)}$. Whenever $|t_2-t_1|\leq \delta$, the following inequality holds:
  \[ |q_j^{n}(t_2)- q_j^{n}(t_1)|  \leq \sqrt{\frac{2d(M_0+1)}{m^{*}}} |t_2-t_1|^{1/2} =\varepsilon.\]
 It implies that for each $j \in [1, N]$, $\{q_j^{n}(t)\}$ is equi-continuous. It follows that the path sequence $\{\mathcal{P}_n=\mathcal{P}_n(q(t))\}$ is equi-continuous. 
  
By the Arzel\`{a}-Ascoli theorem, there exists a subsequence $\{\mathcal{P}_{n_l}\}$ which converges uniformly. The limit $\mathcal{P}_0= \mathcal{P}_0(q(t))$ is in $H^1([0,1], \chi)$ and it satisfies
 \[ \lim_{n_l \to \infty}  \mathcal{P}_{n_l}(t)= \mathcal{P}_0(t), \qquad \text{for all \, } t \in [0,1]. \] 
 In particular,
 \[ \lim_{n_l \to \infty}  \mathcal{P}_{n_l}(0)= \mathcal{P}_0(0), \qquad   \lim_{n_l \to \infty}  \mathcal{P}_{n_l}(1)= \mathcal{P}_0(1).\]
It follows that 
\[ \lim_{n_l \to \infty} a_{{i}_{n_l}} = a_{{i}_{0}}, \qquad  \lim_{n_l \to \infty} b_{{j}_{n_l}} = b_{{j}_{0}}, \quad i=1, \dots, k; \, \,  j=1, \dots, s.\] 
And $\mathcal{P}_0$ satisfies
\[\mathcal{A}(\mathcal{P}_0) = \inf_{  (\vec{\alpha}, \vec{\beta}) \in  \mathcal{S}  } \quad \inf_{ q \in P(Q_s,Q_e) }  \mathcal{A}= \inf_{q \in P(Q_{s_{0}},Q_{e_{0}}) }  \mathcal{A},\]
where $\vec{\alpha}_{0}= (a_{{1}_{0}},\dots, a_{{k}_{0}})$ and $\vec{\beta}_{0}= (b_{{1}_{0}},\dots, b_{{s}_{0}})$. The proof is complete. 

\section{Exclusion of total collisions}\label{exclusiontotalcol}
 In what follows, we concentrate on the minimizing problem
\begin{equation}\label{thispapergoal01}
\inf_{ \{ (a_1,\, a_2, \, b_1, \, b_2) \in  \mathcal{S}  \} }  \quad \inf_{\{ q(0)= Q_{s_1}, \, q(1)= Q_{e_1}, \, q(t) \in H^1([0, 1], \chi) \} }  \mathcal{A},
\end{equation}
where $\mathcal{A}=\int_{0}^{1} (K+U) \, dt$ is defined by \eqref{Action}, 
\begin{equation}\label{Henon101}
 Q_{s_1} = \begin{bmatrix}
-2a_1-a_2   &  0   \\
a_1-a_2  &  0     \\
a_1+2a_2   & 0   
\end{bmatrix}, \qquad \,\,
  Q_{e_1} =  \begin{bmatrix}
0   &  -2b_1    \\
-b_2 &  b_1 \\
b_2   &   b_1
\end{bmatrix},  
\end{equation}
and  $(a_1, \, a_2, \, b_1, \, b_2) \in  \mathcal{S}$, with 
\begin{equation}\label{defofs}
\mathcal{S} = \left\{  a_1 \geq 0, \,\,a_2 \geq 0, \,\,b_1 \in \mathbb{R},\,\, b_2 \in \mathbb{R}  \right\}.  
\end{equation}
By Theorem \ref{Thm1.1}, there exists an action minimizer $\mathcal{P}$, such that
\begin{equation}\label{thispapergoal02}
\mathcal{A}\left( \mathcal{P}\right)=\inf_{ \{ (a_1,\, a_2, \, b_1, \, b_2) \in  \mathcal{S}_1   \} }  \quad \inf_{\{ q(0)= Q_{s_1}, \, q(1)= Q_{e_1}, \, q(t) \in H^1([0, 1], \chi) \} }  \mathcal{A}.
\end{equation}
 Based on the celebrated works of Marchal \cite{Ma1} and Chenciner \cite{CA2}, the main challenge in the existence proof is to exclude possible boundary collisions in $\mathcal{P}$. 

Let $q_i= (q_{ix}, q_{iy}) \, (i=1,2,3)$. By the definition of $Q_{s_1}$, $Q_{e_1}$ in \eqref{Henon101} and $\mathcal{S}$  in \eqref{defofs}, it follows that at $t=0$, the three bodies are on the $x$-axis and satisfy the order $q_{1x}(0) \leq q_{2x}(0) \leq q_{3x}(0)$. The possible collisions at $t=0$ are a binary collision between bodies 1 and 2, a binary collision between bodies 2 and 3 and a total collision. The possible collisions at $t=1$ are a binary collision between bodies 2 and 3 and a total collision. In this section, we define a test path to exclude possible total collisions on both boundaries.

We first find a lower bound of actions for paths with a total collision. We consider the collinear Kepler problem: $\frac{d^2  \gamma}{d t^2} = - \lambda \gamma^{-2}$, where $\lambda>0$ is a constant. There is a unique solution such that $\gamma_{\lambda, 1}(0)= \dot{\gamma}_{\lambda, 1}(1) =0$ and $\dot{\gamma}_{\lambda, 1}(t)>0$ for any $t \in (0, 1)$. This is a degenerate Kepler motion with period $2$. It is known \cite{Gordon} that $\gamma_{\lambda, 1}$ minimizes the action functional for the collinear Kepler motion
\[\int_{0}^{1} \frac{1}{2} \dot{\gamma}^2 + \frac{\lambda}{\gamma} \, dt\]
on $\left\{ \gamma \in H^1([0, 1], \mathbb{R})\,  | \,  \gamma(t)=0 \, \, \text{for some} \, \,  t \in [0, 1] \right\}$. And the minimum action value is
\[ \frac{3}{2} \pi^{2/3} \lambda^{2/3}.  \]
Note that the central configuration of three equal masses can be either a collinear configuration or an equilateral triangle configuration. In the case of collinear configuration,  let $q_1(t)= \gamma(t)$, $q_2(t)=0$ and $q_3(t) = -\gamma(t)$. It follows that
\[ \mathcal{A}_{collinear}=\int_{0}^{1} (K+U) \, dt=   \int_0^1 \dot{\gamma}^2 + \frac{5}{2\gamma}dt= 2  \int_0^1 \frac12\dot{\gamma}^2 + \frac{5}{4\gamma}dt. \]
Hence,
\begin{equation}\label{lowerbddcol}
\mathcal{A}_{collinear} \geq 2 \frac{3}{2} \pi^{\frac{2}{3}}  \left(\frac{5}{4}\right)^{\frac{2}{3}} \approx 7.4672.
\end{equation}
In the case of equilateral triangle configuration, we can set $q_1(t) =\gamma(t)$, $q_2(t) = \gamma(t) R(2\pi/3)$ and $q_3(t) = \gamma(t) R(-2\pi/3)$. It follows that 
\[ \mathcal{A}_{triangle}=\int_{0}^{1} (K+U) \, dt=   \int_0^1 \frac{3}{2}\dot{\gamma}^2 + \frac{3}{ \sqrt{3} \gamma}dt= 3  \int_0^1 \frac12\dot{\gamma}^2 + \frac{1}{\sqrt{3}\gamma}dt. \]
Hence,  
\begin{equation}\label{lowerbddtri}
\mathcal{A}_{triangle} \geq 3 \frac{3}{2} \pi^{\frac{2}{3}}  \left(\frac{1}{\sqrt{3}}\right)^{\frac{2}{3}} \approx 6.6927.
\end{equation}
Inequalities \eqref{lowerbddcol} and \eqref{lowerbddtri} imply that the lower bound of action for paths with total collision is about 6.6927. Therefore, if we can find a test path $ \mathcal{P}_{test} \in H^1([0,1], \chi)$ with action strictly less than $6.69$, then there is no total collision in the minimizer $\mathcal{P}$ of \eqref{thispapergoal01}.

By \cite{Ven1}, a piece of the Schubart orbit can be characterized as an action minimizer between two collinear configurations: 
\[q(0)=Q_{s_2}=  \begin{bmatrix}
-c_1     \\
-c_1     \\
2c_1    
\end{bmatrix}, \qquad \,\,
  q(1)=Q_{e_2} =  \begin{bmatrix}
0     \\
-d_1  \\
d_1   
\end{bmatrix},  \]
with $(c_1, \, d_1) \in  \mathbb{R}^2$.  Our test path is defined by a 1-dimensional path connecting $Q_{s_2}$ and $Q_{e_2}$, which can be seen as an approximation of the Schubart orbit. The definition of the test path $\hat{q}(t)= \begin{bmatrix}
\hat{q_{1x}}(t)  &  0\\
\hat{q_{2x}}(t)   &  0\\
\hat{q_{3x}}(t)  &  0
\end{bmatrix}$ is as follows:
\begin{eqnarray}\label{q1test}
\hat{q}_{1x}(t)=
\begin{cases}
-\frac{9}{10}+ t^{\frac{2}{3}},       & \text{if} \quad    0 \leq t  \leq \frac{1}{8}, \\
\frac{26}{35}t-   \frac{26}{35},          & \text{if} \quad\frac18 \leq t \leq 1,
\end{cases}
\end{eqnarray}
\begin{eqnarray}\label{q2test}
\hat{q}_{2x}(t)=
\begin{cases}
-\frac{9}{10}- t^{\frac{2}{3}},       &  \text{if} \quad  0 \leq t  \leq \frac{1}{8}, \\
-\frac{26}{35}t-   \frac{37}{35},          &\text{if} \quad  \frac18 \leq t \leq 1,
\end{cases}
\end{eqnarray}
and $\hat{q}_{3x}(t) \equiv \frac{9}{5}$.  The path $\hat{q} \in H^1([0,1], \chi)$ and it satisfies the boundary conditions 
\[ \hat{q}(0) \in  \left\{Q_{s_1} \, \big| \,  (a_1, \, a_2, \, b_1, \, b_2)  \in  \mathcal{S} \right\}, \quad  \hat{q}(1) \in \left\{Q_{e_1} \, \big| \,  (a_1, \, a_2, \, b_1, \, b_2)  \in  \mathcal{S} \right\}. \]

\begin{lemma}\label{testpathact}
Assume $\hat{q}(t) \in H^1([0,1], \chi)$ is defined by \eqref{q1test} and \eqref{q2test}, then
\[\mathcal{A}(\hat{q}(t)) < 3.5383.\]
\end{lemma}
\begin{proof}
The action of $\hat{q} \in H^1([0,1], \chi)$ is calculated in two parts. When $t \in [0, 1/8]$, the action of $\hat{q}(t)$ is
\begin{eqnarray}\label{A1est}
\mathcal{A}_1&=& \int_0^{1/8}  L(\dot{\hat{q}}, \hat{q})dt \\ \nonumber
&=& \int_0^{0.125} \frac{4}{9} t^{-2/3} + \frac{1}{2 t^{2/3}} + \frac{1}{2.7- t^{2/3}}+  \frac{1}{2.7+ t^{2/3}} dt\\  \nonumber
& < & \frac{3}{2} \left( \frac{4}{9}+ \frac{1}{2}\right) + 0.125 \times \frac{5.4}{ 2.7^2- (0.125)^{4/3}} \\  \nonumber
&\approx& 1.5100.
\end{eqnarray}
When $t \in [1/8, 1]$, the aciton of $\hat{q}(t)$ is
\begin{eqnarray}\label{A2est}
&&\mathcal{A}_2= \int_{1/8}^1  L(\dot{\hat{q}}, \hat{q})dt \\  \nonumber
&=& \frac{7}{8}\left(\frac{26}{35}\right)^2 + \int_{0.125}^1 \frac{35}{52t+11} + \frac{35}{26t+100}+ \frac{35}{ 89-26t} dt\\   \nonumber
&=&   \frac{7}{8}\left(\frac{26}{35}\right)^2 + \frac{35}{52} \ln(52t+11)\bigg|_{0.125}^1+ \frac{35}{26} \ln (26t+100) \bigg|_{0.125}^1 - \frac{35}{26} \ln(89-26t) \bigg|_{0.125}^1\\   \nonumber
&\approx& 2.0281.
\end{eqnarray}
Therefore, by equations \eqref{A1est} and \eqref{A2est}, the action of $\hat{q}(t) \, (t \in [0,1])$ satisfies
\[  \mathcal{A}(\hat{q}(t)) = \mathcal{A}_1+  \mathcal{A}_2 < 1.5101+2.0282=3.5383. \]
The proof is complete.
\end{proof}

\begin{corollary}\label{boundforAP}
The minimizing path $\mathcal{P}$ has no triple collision and 
\[  \mathcal{A}(\mathcal{P}) < 3.5383. \]
\end{corollary}
\begin{proof}
The definition of $\hat{q}(t)\,  (t \in[0,1])$ implies that $\hat{q}(t) \in H^1([0,1], \chi)$ and 
\[\hat{q}(0) = \begin{bmatrix}
-0.9  &  0\\
-0.9  &  0\\
1.8  & 0 
\end{bmatrix} \in \left\{Q_{s_1} \, \big| \,  \{a_1, \, a_2,\, b_1,\, b_2\} \in \mathcal{S}\right\}, \]
\[  \hat{q}(1) = \begin{bmatrix}
0  &  0\\
-1.8  &  0\\
1.8  & 0 
\end{bmatrix} \in \left\{Q_{e_1} \, \big| \, \{a_1, \, a_2, \, b_1, \, b_2\} \in \mathcal{S}\right\}, \]
where $Q_{s_1}$ and $Q_{e_1}$ are defined by \eqref{Henon101} and $\mathcal{S}$ is defined by \eqref{defofs}. By Lemma \ref{testpathact}, the action of the test path $\hat{q}(t) \in H^1([0,1], \chi)$ is strictly less than 3.5383. It follows that 
\[ \mathcal{A}(\mathcal{P}) < 3.5383. \]
By inequalities \eqref{lowerbddcol} and \eqref{lowerbddtri}, the lower bound of actions of paths containing total collision is  6.6927, which is strictly greater than 3.5383. Hence, the action minimizer $\mathcal{P} \in H^1([0,1], \chi)$ has no total collision. The proof is complete.
\end{proof}

Let a Schubart orbit have a minimal period $4$. At $t=0$, we assume this orbit starts with a binary collision between bodies 1 and 2, and at $t=1$, let body 1 be at $0$ and bodies 2 and 3 be symmetrically located on the two sides. It is clear that $\hat{q}(t) \, (t \in [0,1])$is always on the $x$-axis and its boundaries satisfy $\hat{q}(0) \in Q_{s_2}$ and $\hat{q}(1) \in Q_{e_2}$, where
\[Q_{s_2}=  \begin{bmatrix}
-c_1     \\
-c_1     \\
2c_1    
\end{bmatrix}, \qquad \,\,
  Q_{e_2} =  \begin{bmatrix}
0     \\
-d_1  \\
d_1   
\end{bmatrix},  \]
with $(c_1, \, d_1) \in  \mathbb{R}^2$. According to the variational property of the Schubart orbit in \cite{Ven1}, the action of the test path $\hat{q}(t) \, (t \in [0,1])$ is greater than the action of the Schubart orbit in $[0,1]$. Hence, the following corollary holds.
\begin{corollary}
Let the Schubart orbit have period 4 and denote its action in one period $(t \in [0,4])$ by $\mathcal{A}_{Schubart}$. Then 
\[\frac{1}{4} \mathcal{A}_{Schubart} <  3.5383.  \]
\end{corollary}

\section{No binary collision at $t=1$}\label{revisitkepler}
In this section, we study the possible binary collisions in the minimizer $\mathcal{P} \in H^1([0,1], \chi)$ of \eqref{thispapergoal01}. Note that at $t=0$, the boundary set $\left\{ Q_{s_1} \,\big| \, a_1 \geq 0, a_2 \geq 0 \right\}$ in \eqref{Henon101} has an order restriction $q_{1x}(0) \leq q_{2x}(0)\leq q_{3x}(0)$. At $t=1$, $\left\{Q_{e_1} \, \big| \, b_1, \, b_2 \in \mathbb{R}\right\}$ in \eqref{Henon101} is a two dimensional vector space. A standard local deformation argument can be applied to show that $Q_{e_1}$ contains no binary collision. However, it is shown \cite{Chen4, Yu} that there is no way to exclude binary collisions at $t=0$ by the same argument.

We consider a one-end free boundary value problem in the Kepler problem. Let $q_i=( q_{ix}, q_{iy})$ be the position of mass $m_i \, (i=1,2)$. Let $\mathbf{r}(t)= q_1(t)-q_2(t)$, $\alpha=m_1m_2$ and $\mu=\frac{m_1 m_2}{m_1+m_2}$.
Assume the free-end is at $t=0$. Set the two masses $m_1$ and $m_2$ to be on the $x-$axis with a given order $\left(q_{1x}(0) \leq q_{2x}(0) \right)$ at $t=0$. When $q_{1}(0) = q_{2}(0)$, a parabolic ejection solution can be defined as follows $\mathbf{r}_1(t)= \gamma_0 t^{2/3} \vec{c}$, where $\gamma_0=(\frac{9 \alpha}{ 2 \mu})^{1/3}>0$ is a constant and $|\vec{c}|=1$.

Let $\vec{s}= (-1, 0)$. For given $\epsilon>0$ and a unit vector $\vec{c}$, we fix the vector $\mathbf{r}(\epsilon)=  \gamma_0 \epsilon^{2/3} \vec{c}$, while $\mathbf{r}(0)=q_1(0)-q_2(0)=(q_{1x}(0) - q_{2x}(0), 0)= |\mathbf{r}(0)| \vec{s} \equiv r_0 \vec{s}$ satisfies that $r_0 \equiv  |\mathbf{r}(0)|\geq 0$.  The Lagrangian action of the two-body problem is
\begin{equation}\label{actionkepler}
I(\mathbf{r}(t), \dot{\mathbf{r}}(t))=  \int_0^{\epsilon} \frac{\mu}{2} |\dot{\mathbf{r}}|^2 +   \frac{\alpha}{|\mathbf{r}|}dt.
\end{equation}
We define the set $V$ as follows
\begin{equation*}
V= \bigg\{ \mathbf{r}(t) \in H^1([0, \epsilon], \mathbb{R}^2) \, \bigg| \, \mathbf{r}(0)=|\mathbf{r}(0)| \vec{s}\equiv r_0 \vec{s}, \, \mathbf{r}(\epsilon)= \gamma_0 \epsilon^{\frac2 3} \vec{c}, \,   r_0 \geq 0 \bigg\}.
 \end{equation*}
After fixing the unit vector $\vec{c}$ and the positive constant $\epsilon$, we consider the following one-end free boundary value problem:
\begin{equation}\label{actionminkepler}
\inf_{ \mathbf{r}(t) \in V}   I(\mathbf{r}(t), \dot{\mathbf{r}}(t)).
\end{equation}
Standard results \cite{Chen4, Fusco, Yu} imply that  if $<\vec{c}, \vec{s}>  \neq  -1$, then there exists small enough $\epsilon>0$, such that the parabolic ejection solution is not the action minimizer of \eqref{actionminkepler}.
\begin{lemma}[\cite{Chen4, Fusco, Yu}]\label{keplerlambertresult}
Let $\vec{s}=(-1, 0)$ and $\vec{c}$ be a given unit vector which satisfies $\langle\vec{c},\vec{s}\rangle \neq -1$. Then there exists small enough $\epsilon>0$, such that the parabolic ejection solution $\mathbf{r}_1(t)=\gamma_0 \, t^{\frac23}\vec{c}$ of the Kepler problem does not minimize the action of \eqref{actionkepler}  in the Sobolev space \[V=\left\{ \mathbf{r}\in H^1([0,\epsilon],\mathbb{R}^2) \, \bigg| \, \mathbf{r}(0)=r_0\vec{s},\, \,\, \mathbf{r}(\epsilon)=\gamma_0 \,  \epsilon^{\frac23}\vec{c}, \, \,\, r_0 \geq 0 \right\}.\]
\end{lemma}
Next we show that under certain boundary conditions, the action minimizer $\mathcal{P}_0$ in Theorem \ref{Thm1.1} has no binary collision.  The proof is based on the standard blow up technique. As its application, we can then show that the minimizer $\mathcal{P}$ of \eqref{thispapergoal01} has no binary collision at $t=1$.  

Recall that 
 \begin{equation*}
Q_s= \begin{bmatrix}
q_1(a_1, \dots a_k) \\
\dots \\
q_N(a_1, \dots a_k) 
\end{bmatrix}, \qquad  Q_e= \begin{bmatrix}
q_1(b_1, \dots b_s) \\
\dots \\
q_N(b_1, \dots b_s) 
\end{bmatrix}, 
\end{equation*}
where $Q_s, Q_e \in \chi$, $q_i \in \mathbb{R}^d, (i=1, \dots, N)$ and $\vec{\alpha}=(a_1, \dots, a_k), \, \vec{\beta}=(b_1, \dots, b_s)$ are independent variables. Let $\vec{\alpha} \in \mathcal{S}_1$,  $\vec{\beta} \in  \mathcal{S}_2$, where $ \mathcal{S}_1 \subset \mathbb{R}^k$ and $ \mathcal{S}_2 \subset \mathbb{R}^s$ are closed subsets.  $\mathcal{S}_1$ and $\mathcal{S}_2$ have either finitely many lines as its boundary or no boundary. Let $\mathcal{S}_1 \cup\mathcal{S}_2=\mathcal{S}$. Assume that both $\left\{Q_s \, \big|\,  \vec{\alpha}  \in  \mathbb{R}^k \right\}$ and $\left\{Q_e  \, \big| \, \vec{\beta}  \in  \mathbb{R}^s \right\}$ are linear subspaces.  

The following blow-up results are needed in proving that $\mathcal{P}$ is free of binary collisions at $t=1$. It is known that the bodies involved in a partial collision or a total collision will approach a set of central configurations. More information can be known if the solution under concern is an action minimizer:
\begin{lemma}[Theorem 4.1.18 in \cite{Ven}, or Sec. 3.2.1 in \cite{CA2}]\label{Vench}
If a minimizer $q$ of the fixed-ends problem on time interval $[\tau_1, \tau_2]$ has an isolated collision of $k\leq N$ bodies, then there is a parabolic homethetic collision-ejection solution $\hat{q}$ of the $k-$body problem which is also a minimizer of the fixed-ends problem on $[\tau_1, \tau_2]$.
\end{lemma}

\begin{lemma}[Proposition 5 in \cite{Chen3}, or Sec. 7 in \cite{FT}]\label{blowup}
Let $X$ be a proper linear subspace of $\mathbb{R}^d$. Suppose a local minimizer $x$ of $\mathcal{A}_{t_0,  \, t_1}$ on $B_{t_0, \, t_1}(x(t_0), X):= \{ x \in H^1([t_0, t_1], (\mathbb{R}^d)^N) \, \big|\,  x(t_0) \, \text{is fixed, and} \, \, x_i(t_1) \in X, \, i=1,2, \dots, N \}$ has an isolated collision of $k \leq N$ bodies at $t=t_1$. Then there is a homothetic parobolic solution $\bar{y}$ of the $k$-body problem with $\bar{y}(t_1)=0$ such that $\bar{y}$ is a minimizer of $\mathcal{A}^{*}_{\tau,  \, t_1}$ on $B_{\tau, \,  t_1}(\bar{y}(\tau), X)$ for any $\tau<t_1$. Here $\mathcal{A}^{*}_{\tau, \, t_1}$ denotes the action of this $k$-body subsystem.
\end{lemma}

By applying Lemma \ref{Vench} and Lemma \ref{blowup}, the following theorem holds.
\begin{theorem}\label{nobinarycollision}
Let $\mathcal{S} = \mathbb{R}^{k+s}$. Assume that both $\left\{Q_s \, \big| \,  \vec{\alpha}  \in  \mathbb{R}^k \right\}$ and $ \{Q_e  \, \big|\,  \vec{\beta}  \in  \mathbb{R}^s \}$ are linear subspaces. If the intersection of the two configuration subsets is at origin, i.e. 
\[ \left\{Q_s \, \big| \, \vec{\alpha}  \in  \mathbb{R}^{k}\right\} \cap  \{Q_e  \, \big| \, \vec{\beta}  \in \mathbb{R}^{s}  \}=\{\vec{0} \},\] then the action minimizer $\mathcal{P}_0 \in H^1([0,1], \chi)$ in Theorem \ref{Thm1.1} has no binary collision.
\end{theorem}
\begin{proof}
The existence of action minimizer $\mathcal{P}_0 \in H^1([0,1], \chi)$ is shown by Theorem \ref{Thm1.1}.  By using the results of Marchal \cite{Ma1} and Chenciner \cite{CA2} regarding to minimizing problems with fixed ends, it follows that $\mathcal{P}_0 \in H^1([0, 1], \chi)$ has no collision in $(0,1)$. Then we only need to show that $\mathcal{P}_0\in H^1([0,1], \chi)$ has no binary collision on the two ends. Basically, it has two types of binary collisions: a single binary collision, a simultaneous binary collision. Here we only show the case for single binary collisions by contradiction. Due to \cite{Sarri, Simo, Yan}, simultaneous binary collision can be treated as several separated binary collisions. Then we can exclude the  simultaneous binary collision in $\mathcal{P}_0\in H^1([0,1], \chi)$ similarly. The following argument is standard and it can be found in \cite{CA2, Chen3, Chen, FT, Ma1,  TV} etc.

Assume that $q_1$ and $q_2$ collide at $t=0$. By the analysis of blow up in Lemma \ref{Vench} and Lemma \ref{blowup}, there exists a parobolic homothetic solution $q_i(t)=\xi_i t^{\frac23}, \, (i=1,\, 2)$, which is also a minimizer of the $2$-body problem on $[0, \tau]$ for any $\tau>0$. In fact, $(\xi_1,\, \xi_2)$ forms a central configuration with $m_2\xi_2=-m_1\xi_1$, and the two vectors $\xi_1$, $\xi_2$ satisfy the energy constraint:
\[\sum_{i=1,2} \frac{1}{2} |\frac{2}{3} m_i\xi_i|^2 - \frac{m_1m_2}{|\xi_1-\xi_2|}=0. \]
For a given $\epsilon>0$ small enough, we fix $q_i(\epsilon) \, (i=1,2)$. Next, we perturb $q_i$ to $\bar{q}_i$\, (i=1,2) such that $\bar{q}_i(\epsilon)= q_i(\epsilon) , (i=1,2)$, $\bar{q}_1(0) \neq \bar{q}_2(0)$ and $\bar{q}_i(0) \, (i=1,2)$ satisfy the boundary condition $\left\{Q_s \, \big| \, \vec{\alpha}  \in  \mathbb{R}^{k}\right\}$.
Let 
\[ \overrightarrow{\bar{q}_2\bar{q}_1(0)}= \frac{\bar{q}_1(0)- \bar{q}_2(0)}{|\bar{q}_1(0)- \bar{q}_2(0)|},\]
where $\bar{q}_1(0)$ and $\bar{q_2}(0)$ are the perturbed vectors of $q_1$ and $q_2$ at $t=0$. 

Note that in the linear space $Q_S \equiv \left\{Q_s \, \big| \, \vec{\alpha}  \in  \mathbb{R}^{k}\right\}$, if the binary collision $q_1(0)=q_2(0)$ is in the linear space $Q_S$, there are always two directions to locally deform $q_1$ and $q_2$ at $t=0$: deform $q_i(t)$ to $\bar{q}_i(t) \,(i=1,2)$; or deform $q_1(t)$ to $\bar{q}_2(t)$ and deform $q_2(t)$ to $\bar{q}_1(t)$ for all $t \in [0, \epsilon]$. It follows that one can always choose the local deformation $\bar{q}_i$ such that $\bar{q}_i(0) \, (i=1,2)$ satisfies
\[ < \overrightarrow{\bar{q}_2\bar{q}_1(0)},  \frac{\xi_1}{|\xi_1|}  >  \neq -1.\]
By Lemma \ref{keplerlambertresult}, there exist $\bar{q}_1$ and $\bar{q}_2$, such that the action of $\bar{q}_1$ and $\bar{q}_2$ is strictly smaller than the action of the parabolic ejection solution: $q_1$ and $q_2$. Contradiction! 

Therefore, there is no binary collision in the minimizer $\mathcal{P}_0$. The proof is complete. 
\end{proof}

In the end of this section, we apply Theorem \ref{nobinarycollision} to our minimizing problem:
\begin{equation}\label{minimizerhenon}
\mathcal{A} \left( \mathcal{P} \right)= \inf_{ \{ (a_1,\, a_2, \, b_1, \, b_2) \in  \mathcal{S}_1   \} }  \quad \inf_{\{ q(0)= Q_{s_1}, \, q(1)= Q_{e_1}, \, q(t) \in H^1([0, 1], \chi) \} }  \mathcal{A},
\end{equation}
where $\mathcal{A}=\int_{0}^{1} (K+U) \, dt$, 
\begin{equation}\label{Henon1001}
 Q_{s_1} = \begin{bmatrix}
-2a_1-a_2   &  0   \\
a_1-a_2  &  0     \\
a_1+2a_2   & 0   
\end{bmatrix}, \qquad \,\,
  Q_{e_1} =  \begin{bmatrix}
0   &  -2b_1    \\
-b_2 &  b_1 \\
b_2   &   b_1
\end{bmatrix},  
\end{equation}
and  $(a_1, \, a_2, \, b_1, \, b_2) \in  \mathcal{S}$, with 
\begin{equation}\label{setdef}
\mathcal{S} = \left\{  a_1 \geq 0, \, \, a_2 \geq 0, \, \, b_1 \in \mathbb{R}, \, \,  b_2 \in \mathbb{R}  \right\}.  
\end{equation}
Note that the argument in Theorem \ref{nobinarycollision} can be applied to show that the minimizer $\mathcal{P}$ has no binary collision at $t=1$.  However, at $t=0$, the two possible binary collisions can NOT be directly excluded by the same argument. Assume body $i$ has mass $m_i$ with coordinate $q_i \, (i=1,2,3)$. Let $\mathbf{r}_{ij}=q_i-q_j$. In the three-body problem, it is known \cite{Yu2} that $\displaystyle \lim_{t \to 0^{+}}\frac{\mathbf{r}_{ij}(t)}{|\mathbf{r}_{ij}(t)|}$ exists if a binary collision between $m_i$ and $m_j$ occurs at $t=0$. We define the direction of collision $\vec{c}_{ij}$ to be
\[\vec{c}_{ij}=\lim_{t \to 0^{+}}\frac{\mathbf{r}_{ij}(t)}{|\mathbf{r}_{ij}(t)|}.\]
 
Actually Lemma \ref{keplerlambertresult} implies that
\begin{corollary}\label{180degreeejection}
If the minimizer $\mathcal{P}$ in \eqref{Henon1001} has a binary collision between bodies 1 and 2 at $t=0$, then the direction of collision $\vec{c}_{12}$  must be $(1,0)$.  If it has a binary collision between bodies 2 and 3 at $t=0$, then the direction of collision $\vec{c}_{23}$  must be $(1,0)$. 
\end{corollary}
\begin{proof}
Since $a_1\geq 0$ and $a_2 \geq 0$, it implies that the three bodies lie on the $x-$axis at $t=0$ with an order $q_{1x} (0) \leq q_{2x}(0) \leq q_{3x}(0)$. If bodies 1 and 2 collide at $t=0$, then  
$\vec{s}=(-1, 0)$. If $\vec{c}_{12} \neq (1, 0)$, that is, $\langle\vec{c}_{12},\vec{s}\rangle \neq -1$, Lemma \ref{keplerlambertresult} and the proof of Theorem \ref{nobinarycollision} imply that there exists a local deformation which can lower the action of $\mathcal{P}$. Contradiction! 
Therefore, if there is a binary collision between bodies 1 and 2 in the minimizer $\mathcal{P}$, $\vec{c}_{12}$ must be $(1,0)$. Similarly, if bodies 2 and 3 collide at $t=0$ in $\mathcal{P}$,  $\vec{c}_{23}$ must be $(1,0)$. The proof is complete.
\end{proof}

\section{Analysis of binary collisions at $t=0$}\label{binarycollisionat0}
In this section, we prove that in the minimizer $\mathcal{P}$ in \eqref{Henon1001}, the only possible collision is the binary collision between bodies 1 and 2 at $t=0$. And the corresponding solution must be the Schubart orbit in Fig.~\ref{Schubartorbit}. 
 
\subsection{Exclusion of binary collision between bodies 2 and 3 at $t=0$}
In this subsection, we show that if bodies 2 and 3 collide at $t=0$ in the minimizer $\mathcal{P}$, the action $\mathcal{A}(\mathcal{P}) \geq 4.21617$. By Lemma \ref{testpathact}, there exists a test path $\hat{q}(t)$ which has action $\mathcal{A}(\hat{q}(t)) < 3.5383$.  Therefore, $\mathcal{P}$ has no collision between bodies 2 and 3 at $t=0$. 
\begin{lemma}\label{lowerbdd23collide}
The minimizer $\mathcal{P}$ has no binary collision between bodies 2 and 3 at $t=0$.  
\end{lemma}
\begin{proof}
Assume that in $\mathcal{P}$, bodies 2 and 3 collide at $t=0$. Note that $\mathcal{P}$ is the action minimizer of \eqref{minimizerhenon}, with $Q_{s_1}$, $Q_{e_1}$ and $\mathcal{S}$ defined by \eqref{Henon1001} and \eqref{setdef}. It follows that $\mathcal{P}$ is also an action minimizer under the following settings:
\begin{equation}\label{minprob2}
\inf_{ \{ (a_1, \, b_1, \, b_2) \in  \mathcal{S}_{20}  \} }  \quad \inf_{\{ q(0)= Q_{s_2}, \, q(1)= Q_{e_2}, \, q(t) \in H^1([0, 1], \chi) \} }  \mathcal{A},
\end{equation}
where $\mathcal{A}=\int_{0}^{1} (K+U) \, dt$, 
\begin{equation}\label{Qs2Qe2}
 Q_{s_2} = \begin{bmatrix}
-2a_1   &  0   \\
a_1  &  0     \\
a_1   & 0   
\end{bmatrix}, \qquad \,\,
  Q_{e_2} =  \begin{bmatrix}
0   &  -2b_1    \\
-b_2 &  b_1 \\
b_2   &   b_1
\end{bmatrix},  
\end{equation}
and  $(a_1, \, b_1, \, b_2) \in  \mathcal{S}_{20}$, with 
\begin{equation}\label{defofs2}
\mathcal{S}_{20} = \left\{  a_1 \geq 0, \quad b_1 \in \mathbb{R}, \quad b_2 \in \mathbb{R}  \right\}.  
\end{equation}
By assumption, the action $\mathcal{A}(\mathcal{P})$ is the minimum action for the minimizing problem \eqref{minprob2} with $Q_{s_2}$, $Q_{e_2}$ and $\mathcal{S}_{20}$ defined in \eqref{Qs2Qe2} and \eqref{defofs2}. To find a lower bound of $\mathcal{A}(\mathcal{P})$, we define a new functional space such that  $\mathcal{P}$ is in this space. Let 
\begin{equation}\label{Qs3Qe3}
 Q_{s_3} = \begin{bmatrix}
-2a_1   &  0   \\
a_1  &  c_1     \\
a_1   & -c_1   
\end{bmatrix}, \qquad \,\,
  Q_{e_3} =  \begin{bmatrix}
0   &  -2b_1    \\
-b_2 &  b_1 \\
b_2   &   b_1
\end{bmatrix},  
\end{equation}
and  $(a_1, \, c_1, \, b_1, \, b_2) \in  \mathbb{R}^4$. We consider the following minimizing problem
\begin{equation}\label{minprob3}
\inf_{ \{ (a_1, \, c_1, \, b_1, \, b_2) \in  \mathbb{R}^4  \} }  \quad \inf_{\{ q(0)= Q_{s_3}, \, q(1)= Q_{e_3}, \, q(t) \in H^1([0, 1], \chi) \} }  \mathcal{A}.
\end{equation}
By Theorem \ref{Thm1.1}, there exists a minimizer $\widetilde{\mathcal{P}}$, such that 
\[ \mathcal{A}(\widetilde{\mathcal{P}})=  \inf_{ \{ (a_1, \, c_1, \, b_1, \, b_2) \in  \mathbb{R}^4  \} }  \quad \inf_{\{ q(0)= Q_{s_3}, \, q(1)= Q_{e_3}, \, q(t) \in H^1([0, 1], \chi) \} }  \mathcal{A}.\]
Note that under the boundary settings, it is clear that $\mathcal{A}(\mathcal{P}) \geq \mathcal{A}(\widetilde{\mathcal{P}})$. The rest of the proof is to find a lower bound for $\mathcal{A}(\widetilde{\mathcal{P}})$.\\
By Theorem \ref{nobinarycollision}, the minimizer $\widetilde{\mathcal{P}}$ has no binary collision. Since the circular Lagrangian orbit satisfies the boundary conditions $Q_{s_3}$ and $Q_{e_3}$, it follows that 
\begin{equation}\label{lagrangianact}
\mathcal{A}( \widetilde{\mathcal{P}}  )  \leq  \frac{1}{4} \mathcal{A}_{Lagrangian}= 3\times \frac{3}{2} (2 \pi)^{\frac{2}{3}} \left( \frac{\sqrt{3}}{3}\right)^{\frac{2}{3}} 4^{-\frac{2}{3}} \approx 4.21617,
\end{equation}
where $\mathcal{A}_{Lagrangian}$ is the action of one period of the Lagrangian circular orbit with period 4.
By inequality \eqref{lowerbddtri} in Section \ref{exclusiontotalcol}, the actions of the paths of total collision have a lower bound $\mathcal{A}_{totalcollision} \geq 6.6927$. It implies that $\widetilde{\mathcal{P}}$ has no total collision. Therefore, the minimizer $\widetilde{\mathcal{P}}$ is a solution of the planar three-body problem.

We will show that $\widetilde{\mathcal{P}} \, (t \in [0,1])$ can be extended to a periodic solution with a period $T=4$. Let $\widetilde{q}_i(t)= (q_{ix}(t), \, q_{iy}(t)) \, (i=1, 2, 3, \, t\in [0,1])$ be the coordinates of each body in $\widetilde{\mathcal{P}}$. By formulas of first variation, it is known that 
\begin{equation}\label{velocityat0Qs3}
\dot{\widetilde{q}}_{1x}(0)=0, \qquad \dot{\widetilde{q}}_{2x}(0)= -\dot{\widetilde{q}}_{3x}(0),\qquad \dot{\widetilde{q}}_{2y}(0)=\dot{\widetilde{q}}_{3y}(0);
\end{equation}
\begin{equation}\label{velocityat1Qs3}
 \dot{\widetilde{q}}_{1y}(1)=0, \qquad \dot{\widetilde{q}}_{2y}(1)= -\dot{\widetilde{q}}_{3y}(1),\qquad   \dot{\widetilde{q}}_{2x}(1)=\dot{\widetilde{q}}_{3x}(1).
\end{equation}
By the uniqueness of solution of the initial value problem in an ODE system, the minimizer $\widetilde{\mathcal{P}} \, (t \in [0,1])$ can be extended as follows
\begin{eqnarray}\label{extensioninQs3}
\begin{cases}
 \widetilde{q_1}(t)= \left( -\widetilde{q_{1x}}(2-t), \,\, \,  \widetilde{q_{1y}}(2-t)\right),  &  t \in[1,2];  \\
 \vspace{0.1cm}
 \widetilde{q_2}(t)= \left( -\widetilde{q_{3x}}(2-t), \,\, \,  \widetilde{q_{3y}}(2-t)\right),  &  t \in[1,2];  \\
  \vspace{0.1cm}
  \widetilde{q_3}(t)= \left( -\widetilde{q_{2x}}(2-t), \,\, \,  \widetilde{q_{2y}}(2-t)\right),  &  t \in[1,2];  \\
   \vspace{0.1cm}
 \widetilde{q_1}(t)= \left( -\widetilde{q_{1x}}(t-2), \,\, \,  -\widetilde{q_{1y}}(t-2)\right),  &  t \in[2,4];\\
  \vspace{0.1cm}
  \widetilde{q_2}(t)= \left( -\widetilde{q_{2x}}(t-2), \,\, \,  -\widetilde{q_{2y}}(t-2)\right),  &  t \in[2,4];\\
   \vspace{0.1cm}
 \widetilde{q_3}(t)= \left( -\widetilde{q_{3x}}(t-2), \,\, \,  -\widetilde{q_{3y}}(t-2)\right),  &  t \in[2,4].
\end{cases}
\end{eqnarray}
Let $\widetilde{q}(t)= \begin{bmatrix}
\widetilde{q_1}(t)\\
\widetilde{q_2}(t)\\
\widetilde{q_3}(t)
\end{bmatrix}$ and $q(t) = \begin{bmatrix}
q_{1x}(t) & q_{1y}(t)\\
q_{2x}(t) & q_{2y}(t)\\
q_{3x}(t) & q_{3y}(t)
\end{bmatrix}$. In \cite{Long}, Long and Zhang showed that in the planar three-body problem, the action minimizer in the loop space 
\[ \Omega= \left\{ q(t) \in H^1(\mathbb{R}, \chi) \,  \bigg|  \, q(t)= q(t+T), \quad q(t)= -q(t+T/2), \, \, \, \, \forall  \,\, t \in \mathbb{R} \right\}\]
 must be the non-collision equilateral triangle circular periodic solutions. Note that $T=4$. Its action is
 \[\mathcal{A}_{Lagrangian}=3\times \frac{3}{2} (2 \pi)^{\frac{2}{3}} \left(\frac{\sqrt{3}}{3}\right)^{\frac{2}{3}} 4^{\frac{1}{3}}  \approx 16.8647. \]
 Since $\widetilde{q}(t) \in \Omega$, it follows that
 \begin{equation}\label{Alowbddlag}
 \mathcal{A}(\widetilde{q}(t)\, (t \in[0,1]))=  \mathcal{A}(\widetilde{\mathcal{P}}) \geq \frac{1}{4} \mathcal{A}_{Lagrangian} \approx 4.21617.
 \end{equation}
 Inequalities \eqref{lagrangianact} and \eqref{Alowbddlag} imply that 
 \[ \mathcal{A}(\widetilde{\mathcal{P}}) =  \frac{1}{4} \mathcal{A}_{Lagrangian} \approx  4.21617.   \]
Therefore, 
\[  \mathcal{A}(\mathcal{P}) \geq \mathcal{A}(\widetilde{\mathcal{P}})  \approx  4.21617. \]
By Corollary \ref{boundforAP}, the lower bound of the minimizer $\mathcal{P}$ is $ \mathcal{A}(\mathcal{P}) <3.5383<4.21617$. Contradiction!\\
Hence, the minimizer $\mathcal{P}$ has no binary collision between bodies 2 and 3 at $t=0$. The proof is complete.
\end{proof}

\subsection{Binary collision between bodies 1 and 2 at $t=0$}\label{5.2}
By Theorem \ref{nobinarycollision} and Lemma \ref{lowerbdd23collide}, the only possible collision in the minimizer $\mathcal{P}$ is the collision between bodies 1 and 2 at $t=0$. In this subsection, we show that the minimizer $\mathcal{P}$ with a binary collision must coincide with the Schubart orbit \cite{Sch} in Fig.~\ref{Schubartorbit}, which has been shown to exist in \cite{Moe, Ven1, Yan3}.

Assume that bodies 1 and 2 collide at $t=0$ in $\mathcal{P}$. By Corollary \ref{180degreeejection}, the direction of collision $\vec{c}_{12}=\lim_{t \to 0^{+}}\frac{q_1-q_2}{|q_1-q_2|}$ must be $(1,0)$. By Yu's work \cite{Yu2}, more information can be shown about the colliding pair: bodies 1 and 2. After taking away the movement of the center of mass of this colliding pair, the asymptotic behavior of $q_1-q_2$ at $t=0$ satisfies:
\begin{lemma}\label{asympvec}
For $t>0$ small enough, if the $y$ component of $q_1-q_2$ is not $0$, we have
\[q_1-q_2 = ( a t^{2/3}+o(t^{2/3}), \, \,  b t^{\alpha_0}+o(t^{\alpha_0}) ), \]
where $\alpha_0 >\frac{5}{3}$, $a\neq 0, b\neq 0$.
Furthermore, the $y-$velocity satisfies
\[ \lim_{t \to 0^{+}} \left(\dot{q}_{1y}-\dot{q}_{2y} \right)=0.   \]
\end{lemma}
\begin{proof}
The proof is based on Proposition 4.2 in \cite{Yu2}. Actually, it is shown that
\begin{equation}\label{limit1}
 \lim_{t \to 0^{+}} \frac{q_1-q_2}{|q_1-q_2|} = (1, \, 0), 
 \end{equation}
 \begin{equation}\label{limit2}
\lim_{t \to 0^{+}}  \frac{d}{dt} \left( \frac{q_1-q_2}{|q_1-q_2|} \right)= (0, \, 0).
 \end{equation} 
It is known \cite{SK} that an isolated binary collision in the three-body problem can be regularized. By \eqref{limit1}, it implies that for small enough $t>0$, there exists $\alpha_0>2/3$, such that
\[q_1-q_2 = ( a t^{2/3}+o(t^{2/3}), \, \,  b t^{\alpha_0}+o(t^{\alpha_0}) ). \]
Hence, if $b \neq 0$,
\[ \frac{d}{dt} \left( \frac{q_1-q_2}{|q_1-q_2|} \right)= \left( \frac{-ab^2 t^{2 \alpha_0 - 7/3}(3\alpha_0 -2) }{3 (a^2+ b^2 t^{2 \alpha_0 - 4/3})^{3/2}} +o(t^{2 \alpha_0 - 7/3}), \, \,      \frac{-ab^2 t^{ \alpha_0 - 5/3}(3\alpha_0 -2) }{3 (a^2+ b^2 t^{2 \alpha_0 - 4/3})^{3/2}} +o(t^{\alpha_0 - 5/3})\right). \]
By \eqref{limit2}, it follows that
\[ 2 \alpha_0 - 7/3 >0, \quad  \alpha_0 - 5/3 >0. \]
Therefore, 
\[  \alpha_0 > 5/3.  \]
It follows that 
\begin{equation}\label{velocity12ycol}
 \lim_{t \to 0^{+}} \left(\dot{q}_{1y}-\dot{q}_{2y} \right)=0.   
 \end{equation}
If the $y$ component of $q_1-q_2$ is $0$, it is clear that \eqref{velocity12ycol} holds. The proof is complete. 
\end{proof}

In order to prove that the minimizing path $\mathcal{P}$ must coincide with the Schubart orbit, we introduce a geometric argument to study the minimizer $\mathcal{P}$ connecting a collision configuration $Q_{s_4}$ in \[Q_{S_4}\equiv \left\{Q_{s_4}=\begin{bmatrix}
  -a_2  & 0 \\
     -a_2 & 0 \\
    2a_2 & 0
\end{bmatrix} \bigg| \, a_2 \geq 0 \right \}\] and an isosceles triangle configuration $Q_{e_1}$ in \[Q_{E_1}\equiv \left\{Q_{e_1}=\begin{bmatrix}
0   &  -2b_1    \\
-b_2 &  b_1 \\
b_2   &   b_1
\end{bmatrix} \bigg| \, b_1\in \mathbb{R}, b_2\in \mathbb{R} \right\}.\] 
The minimizer $\mathcal{P}$ satisfies
\begin{equation}\label{minPact}
 \mathcal{A}(\mathcal{P})= \inf_{ \{ q \in H^1([0,1], \chi) | q(0) \in Q_{S_4}, \, q(1) \in Q_{E_1} \} } \mathcal{A}. 
 \end{equation}
This geometric argument is first introduced in \cite{YK}, in which we applied it to show that the retrograde double-double orbit in the parallelogram equal-mass four-body problem has a smaller action than the prograde double-double orbit. Furthermore, the same argument can be applied to study the retrograde orbits and the prograde orbits \cite{Chen2, Chen1} in the planar three-body problem with mass $M=[1, \, 1, \, m]$. We can show that for a given period, the retrograde orbit has a smaller action than the prograde orbit. 

Our argument starts with the Jacobi coordinates of the $3$-body problem with equal masses:
\begin{equation*}
  Z_1=q_1-q_2,\ \ \ \ Z_2=q_3-\frac{q_1+q_2}{2}.
\end{equation*}
\begin{figure}[ht]
    \begin{center}
   \psset{xunit=2cm,yunit=2cm}
\begin{pspicture}(-1,-1)(1,1)
\psline[linewidth=1pt, linecolor=blue](-0.3,0.3)(-0.07, 0.07)
\psline[linewidth=1pt, linecolor=blue](-0.8,0.15)(0.2,0.45)
\psdots[dotsize=4pt](0.4,-0.4)(-0.8,0.15)(0.2,0.45)(-0.3,0.3)(-0.07, 0.07)
\psaxes[linewidth=1.5pt]{->}(-0.07, 0.07)(-0.95, -0.85)(0.95, 0.85)
\psline[linewidth=1pt, linecolor=blue, linestyle=dashed]{->}(-0.07, 0.07)(0.93, 0.37)
\psline[linewidth=1pt, linecolor=blue, linestyle=dashed]{->}(-0.07, 0.07)(0.63, -0.63)
\rput(1, -0.1){$x$}
\rput(0.1, 0.85){$y$}
\rput(0.98, 0.47){\textcolor{blue}{$Z_1$}}
\rput(0.5, -0.73){\textcolor{blue}{$Z_2$}}
\rput(-0.17, -0.1){0}
\rput(-0.95, 0.25){$q_2$}
\rput(0.38, 0.55){$q_1$}
\rput(0.55, -0.3){$q_3$}
  \end{pspicture}
   \end{center}
 \caption{ \label{Jacobifigure}The Jacobi coordinates, where $Z_1$, $Z_2$ are vectors starting at the origin and they are in dashed lines. }
\end{figure}
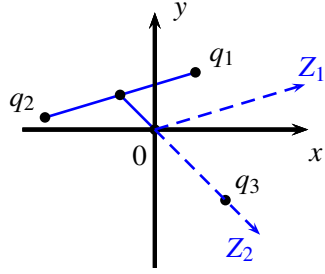  
Because the center of mass is fixed at $0$, the coordinates of $q_i \, (i=1,2,3)$ can be written as
\begin{align*}
  & q_1 = -\frac{1}{3}Z_2+\frac{1}{2}Z_1,  \qquad  q_2 = -\frac{1}{3}Z_2-\frac{1}{2}Z_1, \qquad   q_3 = \frac{2}{3}Z_2.
\end{align*}
The action can then be written as
\begin{align}
  \mathcal{A}(q)&=\int_0^1 \frac{1}{2}\sum_{i=1}^3|\dot{q}_i|^2+U(q) \, dt,\nonumber \\
  &=\int_0^1  K(Z_1, \, Z_2)+U(Z_1, \, Z_2) \, dt \label{actioninJacobi},\\
  &\equiv \mathcal{A}(Z_1,Z_2),\nonumber
\end{align}
where
\begin{equation}\label{Upoten}
K(Z_1, \, Z_2)=\frac{1}{4}|\dot{Z}_1|^2+\frac{1}{3}|\dot{Z}_2|^2, \, \quad \, U(Z_1, \, Z_2)= \frac{1}{|Z_1|}+ \frac{1}{|Z_2+\frac{1}{2} Z_1|}+\frac{1}{|Z_2-\frac{1}{2} Z_1|}.
\end{equation}
By Marchal \cite{Ma1} and Chenciner \cite{CA2} and the argument in Section \ref{revisitkepler}, the minimizing path is collision-free for $t \in (0, 1]$. Hence, $Z_1\not=0$ for $t \in (0,1]$. By the boundary setting in $Q_{S_4}$, it implies that $Z_1(0)=0$ and $Z_2(0)=(3a_2, 0) \not= 0$ is on the positive $x$-axis. 
  
The angle $\beta=\beta(Z_1, \, Z_2)$ between the two vectors $Z_1$ and $Z_2$ can be defined as follows:
\[  \beta= \beta(Z_1, \, Z_2) = \arccos{\frac{< Z_1, \, Z_2 >}{|Z_1|||Z_2|}}, \]
whenever both $Z_1 \not= 0$ and $Z_2 \not= 0$ hold. Then we can define an angle $\Delta\theta= \Delta\theta(Z_1, \, Z_2)$ by
\begin{equation}\label{deltatheta}
  \Delta\theta= \Delta\theta(Z_1, \, Z_2)= \left\{\begin{aligned}
                        & \beta,\ \ \text{if}  \ \ \beta\le\frac{\pi}{2}; \\
                        & \pi-\beta,\ \ \text{if}\ \   \beta>\frac{\pi}{2}.\\
                     \end{aligned}
                     \right.
\end{equation}
It is clear that $\Delta\theta\in [0, \pi/2]$. Geometrically, $\Delta\theta$ describes the acute angle formed by the two straight lines spanned by the two nonzero vectors $Z_1$ and $Z_2$ respectively. 

For convenience, in the Cartesian $xy$ coordinate system, the $i$-th quadrant is denoted by $ \mathsf{Q}_i \, (i=1, 2, 3, 4)$, while its closure is denoted by $\overline{\mathsf{Q}_i}$. For example, $\mathsf{Q}_1=\{ (x, \,  y) \, |\, x>0,\,  y>0 \}$ and $\overline{\mathsf{Q}_1}=\{(x, \, y) \, | \, x\geq 0, \, y \geq 0\}$.

Given $t\in (0,1]$, if $Z_i= Z_i(t) \not=  0, \ (i=1,2)$, a formula of $U(Z_1,\, Z_2)=U(Z_1(t), \, Z_2(t))$ in terms of $|Z_1|$, \, $|Z_2|$ and $\Delta\theta$ can be derived by the law of cosines:
\begin{equation}\label{formulaofUinz1z2theta}
\begin{split}
U(Z_1, \, Z_2)&= \frac{1}{|Z_1|}+ \frac{1}{|Z_2+\frac{1}{2} Z_1|}+\frac{1}{|Z_2-\frac{1}{2} Z_1|} \\
&=  \frac{1}{|Z_1|}+ \frac{1}{\sqrt{\frac{1}{4} |Z_1|^2+ |Z_2|^2 + |Z_1||Z_2| \cos (\Delta \theta) }}+ \frac{1}{\sqrt{\frac{1}{4} |Z_1|^2+ |Z_2|^2 - |Z_1||Z_2| \cos (\Delta \theta) }}. 
\end{split}
\end{equation}
By \eqref{formulaofUinz1z2theta}, $U(Z_1,Z_2)$ is a function of three variables: $|Z_1|$, $|Z_2|$ and $\Delta\theta$ when both $Z_1 \not= 0$ and $Z_2 \not= 0$ hold. Let $U(|Z_1|, \, |Z_2|, \, \Delta\theta) \equiv U(Z_1, \, Z_2)$.  Indeed, $U(|Z_1|, \, |Z_2|, \, \Delta\theta)$ satisfies the following property.

\begin{proposition}\label{Uintheta}
Fix $|Z_1| \neq 0$ and $|Z_2| \neq 0$ and assume that the potential energy $U(|Z_1|,|Z_2|, \Delta\theta)=U(Z_1, \, Z_2)$ in \eqref{formulaofUinz1z2theta} is finite. Then $U(|Z_1|, \, |Z_2|, \, \Delta\theta)$ is a strictly decreasing function with respect to $\Delta\theta$.
\end{proposition}
\begin{proof}
Fixing $|Z_1|$ and $|Z_2|$ and taking the derivative of $U(|Z_1|, \, |Z_2|, \,  \Delta\theta)$ in \eqref{formulaofUinz1z2theta} with respect to $\Delta\theta$, it follows that 
\begin{eqnarray*}
& &\frac{\partial   U(|Z_1|,\, |Z_2|, \, \Delta\theta)}{\partial \Delta\theta}\\
&=& \frac{\frac{1}{2} |Z_1||Z_2|\sin(\Delta\theta )}{\left[\frac{1}{4} |Z_1|^2+ |Z_2|^2 + |Z_1||Z_2| \cos (\Delta \theta) \right]^\frac{3}{2}}- \frac{\frac{1}{2} |Z_1||Z_2|\sin(\Delta\theta )}{\left[\frac{1}{4} |Z_1|^2+ |Z_2|^2 - |Z_1||Z_2| \cos (\Delta \theta) \right]^\frac{3}{2}}.
\end{eqnarray*}
Note that $\Delta\theta \in [0,\frac{\pi}{2}]$. It implies that
\[ \frac{\partial   U(|Z_1|,\, |Z_2|, \, \Delta\theta)}{\partial \Delta\theta} \leq 0, \]
and $\frac{\partial   U(|Z_1|,\, |Z_2|,\,  \Delta\theta)}{\partial \Delta\theta}=0$ if and only if $\Delta\theta= 0$ or $\Delta\theta= \pi/2$. The proof is complete.
%
%
\end{proof}

In our case, if $(Z_1, \, Z_2)$ is the Jacobi coordinate of the minimizer $\mathcal{P}$ in \eqref{minPact}, it follows that $U(Z_1, \, Z_2)=U(Z_1(t), \, Z_2(t))$ is finite for any $t \in (0,1]$. Let $Z_i=(Z_{ix}, \, Z_{iy}) \, (i=1,2)$. Assume that $Z_1(1) \in \overline{\mathsf{Q}_1}$ and $Z_2(1) \in \overline{\mathsf{Q}_4}$. A new path $(\widetilde{Z}_1, \, \widetilde{Z}_2)=(\widetilde{Z}_1(t), \, \widetilde{Z}_2(t))$ is defined by
\begin{equation}\label{tildeZ1Z2def}
\widetilde{Z}_1(t)= (|Z_{1x}(t)|, \, |Z_{1y}(t)|), \qquad  \widetilde{Z}_2(t)= (|Z_{2x}(t)|, \, -|Z_{2y}(t)|), \qquad \, \, \forall \, t \in [0,1].
\end{equation}
A demonstration of $Z_1$ and $\widetilde{Z}_1$ is shown in Fig.~\ref{reflection}.

It is clear that $\, \widetilde{Z}_1(t) \in \overline{\mathsf{Q}_1} \, $, $\, \widetilde{Z}_2(t) \in \overline{\mathsf{Q}_4} \, $ and $|\widetilde{Z}_i(t)|=|Z_i(t)|$ for all $t \in [0,1]$. And $\int_0^1 K(\widetilde{Z}_1, \, \widetilde{Z}_2) \, dt=\int_0^1 K(Z_1, \, Z_2) \, dt$, where $K(Z_1, \, Z_2)= \frac{1}{4}|\dot{Z}_1|^2+\frac{1}{3}|\dot{Z}_2|^2$ is the kinetic energy.

If $Z_1(t) \not=0$ and $Z_2(t) \not=0$, an angle $\Delta \tilde{\theta}=\Delta \tilde{\theta}(\widetilde{Z}_1, \, \widetilde{Z}_2)$ can be defined as follows:
\begin{equation}\label{deltathetatilde}
  \Delta \tilde{\theta}=\Delta \tilde{\theta}(\widetilde{Z}_1, \, \widetilde{Z}_2)=\left\{\begin{aligned}
                        & \tilde{\beta},\ \ \text{if}  \ \ \beta\le\frac{\pi}{2}; \\
                        & \pi-\tilde{\beta},\ \ \text{if}\ \   \beta>\frac{\pi}{2},\\
                     \end{aligned}
                     \right.
\end{equation}
where $\tilde{\beta}= \tilde{\beta}(\widetilde{Z}_1, \, \widetilde{Z}_2)=\arccos{\frac{ \langle \widetilde{Z}_1, \, \widetilde{Z}_2 \rangle}{|\widetilde{Z}_1|||\widetilde{Z}_2|}}$.

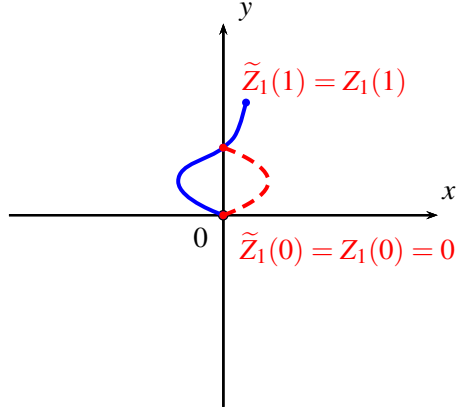
\begin{figure}[ht]
    \begin{center}
\psset{xunit=3cm,yunit=3cm}
\begin{pspicture}(-1,-1)(1,1)
  \psaxes{->}(0, 0)(-0.95, -0.85)(0.95, 0.85)
  \rput(1, 0.1){$x$}
  \rput(0.1, 0.9){$y$}
  \rput(-0.1, -0.1){$0$}

  \psdots[dotsize=4pt](0,0)
  \psdots[dotsize=3pt, linecolor=blue](0, 0)(0, 0.3)(0.1, 0.5)
  \pscurve[linewidth=1.5pt, linecolor=blue](0, 0)(-0.2, 0.15)(0,  0.3)(0.05, 0.35)(0.1, 0.5)
   \psdots[dotsize=3pt, linecolor=red](0, 0)(0, 0.3)
  \pscurve[linewidth=1.5pt, linestyle=dashed, linecolor=red](0, 0)(0.2, 0.15)(0, 0.3)
  
    \rput(0.55, -0.15){\textcolor{red}{$\widetilde{Z}_1(0)= Z_1(0)=0$}}
  \rput(0.45, 0.6){\textcolor{red}{$\widetilde{Z}_1(1)=Z_1(1)$}}
\end{pspicture}
  \end{center}
 \caption{ \label{reflection}A demonstration of $Z_1$ and $\widetilde{Z}_1$, where the blue curve is $Z_1(t) \, (t \in [0,1])$. When $Z_1(t)$ is in the second quadrant, $\widetilde{Z}_1(t)$ is defined to be the dashed red curve which is a reflection of $Z_1(t)$ about the $y$-axis. When $Z_1(t)$ is in the first quadrant, we define $\widetilde{Z}_1(t)=Z_1(t)$. }
 \end{figure}

With the help of Proposition \ref{Uintheta}, we can show that
\begin{lemma}\label{comaprevalueofU}
For any given time $t \in (0,1]$, 
\[U(Z_1(t), Z_2(t)) \geq U(\widetilde{Z}_1(t), \widetilde{Z}_2(t)), \]
where $U(Z_1(t), Z_2(t))=\frac{1}{|Z_1(t)|}+\frac{1}{|Z_2(t)+\frac{1}{2} Z_1(t)|}+\frac{1}{|Z_2(t)-\frac{1}{2} Z_1(t)|}$. The equality holds if and only if $Z_1(t)$ and $Z_2(t)$ are in two adjacent closed quadrants respectively. 

\end{lemma}
\begin{proof}
Note that $Z_1=Z_1(t) \not=0$ for all $t \in (0,1]$. The inequality $U(Z_1(t), Z_2(t)) \geq U(\widetilde{Z}_1(t), \widetilde{Z}_2(t))$ is proved pointwisely in the following two cases. 

\textbf{Case 1:} There exists some time $t_0 \in (0,1]$, such that $Z_2(t_0)=0$. In this case, $\Delta \theta$ is NOT well defined. By \eqref{formulaofUinz1z2theta},
\begin{eqnarray*}
U(Z_1, \, Z_2)(t_0)&=& \frac{1}{|Z_1(t_0)|}+\frac{1}{|Z_2(t_0)+\frac{1}{2} Z_1(t_0)|}+\frac{1}{|Z_2(t_0)-\frac{1}{2} Z_1(t_0)|} \\
&=& \frac{5}{|Z_1(t_0)|}.
\end{eqnarray*}
Note that $|Z_1(t_0)|= |\widetilde{Z}_1(t_0)| \not=0$ and $|Z_2(t_0)|= |\widetilde{Z}_2(t_0)| =0$, it follows that
\[U(\widetilde{Z}_1(t_0),\,  \widetilde{Z}_2(t_0))= \frac{5}{|\widetilde{Z}_1(t_0)|}= \frac{5}{|Z_1(t_0)|}= U(Z_1(t_0), \, Z_2(t_0)). \]
That is, whenever $Z_2(t)=0$, we have
\[U(\widetilde{Z}_1(t), \, \widetilde{Z}_2(t))= U(Z_1(t), \, Z_2(t)).\]
\textbf{Case 2:} If $Z_2\not=0$, we can define $\Delta \theta$ as in \eqref{deltatheta} and $\Delta \tilde{\theta}$ as in \eqref{deltathetatilde}. By \eqref{formulaofUinz1z2theta}, 
\[ U(Z_1, \, Z_2)=  \frac{1}{|Z_1|}+\frac{1}{\sqrt{\frac{1}{4} |Z_1|^2+ |Z_2|^2 + |Z_1||Z_2| \cos (\Delta \theta) }}+ \frac{1}{\sqrt{\frac{1}{4} |Z_1|^2+ |Z_2|^2 - |Z_1||Z_2| \cos (\Delta \theta) }}. \]
We show that if $Z_1=Z_1(t)\not=0$ and $Z_2=Z_2(t)\not=0$ are fixed, then \[\Delta \theta(Z_1, \, Z_2)=\Delta \theta \leq \Delta \tilde{\theta}=\Delta \tilde{\theta}(\widetilde{Z}_1, \, \widetilde{Z}_2). \]
In fact, similar to the definition of $\Delta \theta$, we define two angles $\alpha_1$ and $\alpha_2$ as follows
\begin{equation}\label{alpha1alpha2}
\begin{split}
\alpha_1&=\min \left\{ \arccos \frac{\langle Z_1, \vec{s}_1 \rangle}{|Z_1|}, \, \,  \pi-  \arccos \frac{\langle Z_1, \vec{s}_1 \rangle}{|Z_1|} \right\}, \\
\alpha_2&=\min \left\{ \arccos \frac{\langle Z_2, \vec{s}_1 \rangle}{|Z_2|}, \, \, \pi-  \arccos \frac{\langle Z_2, \vec{s}_1 \rangle}{|Z_2|}  \right\},
\end{split}
\end{equation}
where $\vec{s}_1=(1,0)$. It is clear that $\alpha_1, \alpha_2 \in [0, \pi/2]$. 

If $Z_1=Z_1(t)$ and $Z_2=Z_2(t)$ are in two adjacent quadrants respectively, then 
\begin{equation}\label{Z1Z2Deltaangle1}
\Delta \theta=\Delta \theta(Z_1, \, Z_2)= \min \left\{ \alpha_1+\alpha_2,  \, \, \pi-\alpha_1-\alpha_2 \right\}.
\end{equation}
By the definition of $\widetilde{Z}_1$, $\widetilde{Z}_2$ and $\Delta \tilde{\theta}$ in \eqref{deltathetatilde}, it follows that
\begin{equation}\label{Z1Z2Deltaangle1tilde}
\Delta \tilde{\theta}= \Delta \tilde{\theta}(\widetilde{Z}_1, \, \widetilde{Z}_2)= \min \left\{ \alpha_1+\alpha_2, \, \,  \pi-\alpha_1-\alpha_2 \right\}.
\end{equation}
Hence, $\Delta \theta=\Delta \tilde{\theta}$ when $Z_1=Z_1(t)$ and $Z_2=Z_2(t)$ are in two adjacent quadrants respectively. It follows that
\[ U(Z_1(t), \, Z_2(t)) = U(\widetilde{Z}_1(t), \, \widetilde{Z}_2(t)). \]

If $Z_1=Z_1(t)$ and $Z_2=Z_2(t)$ are not in two adjacent quadrants, we have
\begin{equation}\label{Z1Z2Deltaangle2}
\Delta \theta=\Delta \theta(Z_1, \, Z_2)= |\alpha_1-\alpha_2|.
\end{equation}
On the other hand, the angle $\Delta \tilde{\theta}$ is
\begin{equation}\label{Z1Z2Deltaangle2tilde}
\Delta \tilde{\theta}=\Delta \tilde{\theta}(\widetilde{Z}_1, \, \widetilde{Z}_2)= \min \left\{ \alpha_1+\alpha_2, \, \,  \pi-\alpha_1-\alpha_2 \right\}.
\end{equation}
If $\alpha_1 +\alpha_2 \le\frac{\pi}{2}$, then
\begin{equation*}
  \min\{\alpha_1+\alpha_2, \, \pi-\alpha_1-\alpha_2 \}=\alpha_1+\alpha_2 \ge |\alpha_1-\alpha_2|.
\end{equation*}
If $\alpha_1+\alpha_2>\frac{\pi}{2}$, then
\begin{align*}
  \min\{\alpha_1+\alpha_2, \, \pi-\alpha_1-\alpha_2\}&=\pi-\alpha_1-\alpha_2 \\
   &\ge \frac{\pi}{2}- \min \{\alpha_1, \alpha_2 \}\\
   & \geq  |\alpha_1-\alpha_2|.
\end{align*}
It follows that if $Z_1=Z_1(t)$ and $Z_2=Z_2(t)$ are not in two adjacent quadrants, then $\Delta \tilde{\theta} \geq \Delta \theta$ and the equality holds if and only if $\alpha_1=0$ or $\pi/2$ or $\alpha_2=0$ or $\pi/2$. By Proposition \ref{Uintheta}, we have
\[ U(Z_1(t), \, Z_2(t)) \geq U(\widetilde{Z}_1(t), \, \widetilde{Z}_2(t)). \]
Therefore, by the analysis in the two cases, it follows that
\[ U(Z_1(t), \, Z_2(t)) \geq U(\widetilde{Z}_1(t), \, \widetilde{Z}_2(t)),  \qquad  \quad \forall \, t \in (0,1].\]
And the equality $U(Z_1(t), \, Z_2(t)) = U(\widetilde{Z}_1(t), \, \widetilde{Z}_2(t))$ holds if and only if $Z_1=Z_1(t)$ and $Z_2=Z_2(t)$ are in two adjacent closed quadrants respectively. The proof is complete.
\end{proof}

By Lemma \ref{comaprevalueofU}, the action values $\mathcal{A}(Z_1, \, Z_2)$ and $\mathcal{A}( \widetilde{Z}_1,  \, \widetilde{Z}_2)$ in \eqref{actioninJacobi} of the two paths $(Z_1, \, Z_2)$ and $(\widetilde{Z}_1, \, \widetilde{Z}_2)$ satisfy:
\begin{equation*}\label{compareKandU}
\begin{split}
\int_0^1 K(Z_1, \, Z_2) \, dt &= \int_0^1 K(\widetilde{Z}_1,  \, \widetilde{Z}_2) \, dt, \\
U(Z_1,\,  Z_2) & \geq \,  U(\widetilde{Z}_1, \, \widetilde{Z}_2),  \qquad  \qquad \forall \, \, \,  t \in (0,1].
\end{split}
\end{equation*}
It implies that 
\begin{equation}\label{z1z2tildez1z2action}
\mathcal{A}(Z_1, \, Z_2)  \geq  \mathcal{A}( \widetilde{Z}_1, \,  \widetilde{Z}_2).
\end{equation}
On the other hand, recall that the path $(Z_1,Z_2)\in H_0^1([0,1],\mathbb{R}^4)$ minimizers the action $\mathcal{A}$ in \eqref{minPact}. We assume $Z_1(0)=0$ and $Z_2(0) \not= 0$ is on the positive $x$-axis, and $Z_1(1) \in \overline{\mathsf{Q}_1}$ and $Z_2(1) \in \overline{\mathsf{Q}_4}$. By the definition of $( \widetilde{Z}_1, \,  \widetilde{Z}_2)$ in \eqref{tildeZ1Z2def}, it follows that $( \widetilde{Z}_1,  \,\widetilde{Z}_2) \in \{ q \in H^1([0,1], \chi) | q(0) \in Q_{S_4}, \, q(1) \in Q_{E_1} \}$, which implies that
 $\mathcal{A}(Z_1, \, Z_2)  \leq  \mathcal{A}( \widetilde{Z}_1,  \, \widetilde{Z}_2).$ Therefore, by \eqref{z1z2tildez1z2action}, it follows that
\[  \mathcal{A}(Z_1,\,  Z_2) = \mathcal{A}( \widetilde{Z}_1, \,  \widetilde{Z}_2). \]

Note that by the definition of $Q_{e_1}$, $Z_1(1)= (b_2, \, -3 b_1)$ and $Z_2(1)=(\frac{3}{2}b_2,  \, \frac{3}{2}b_1)$. Since the minimizer $(Z_1,Z_2)$ has no collision between bodies 2 and 3, it follows that $b_2 \not=0$ in $Z_i(1) \, (i=1,2)$. Hence, $Z_i(1) \, (i=1,2)$ can not be on the $y$-axis. 

Next, we show (in Theorem \ref{geometrictheorem}) that the minimizing path $(Z_1,Z_2)$ must be the same as the corresponding path $(\widetilde{Z}_1,  \, \widetilde{Z}_2)$ if either $Z_1(1)$ is strictly inside the first quadrant (i.e. not on the axes), or $Z_2(1)$ is strictly inside the fourth quadrant (i.e. not on the axes).
\begin{proposition}\label{z1z2sametimetangent}
If there exists some $t_0 \in (0, 1]$, such that both $Z_1(t_0)$ and $Z_2(t_0)$ are tangent to the axes, then both $Z_1$ and $Z_2$ must stay on the corresponding axes for all $t \in (0, 1]$.
\end{proposition}
\begin{proof}
 The proof basically follows by the uniqueness of solution of the initial value problem of an ODE system. Note that for $t \in (0, 1]$, $q_i \, (i=1,2,3)$ are the solutions of the Newtonian equations. Without loss of generality, we assume $Z_1(t_0)$ is tangent to the $x$-axis. Note that $Z_1(t_0) \neq 0$ and $Z_1=q_1-q_2$. It follows that
\begin{equation}\label{q1ya2yt0equal}
 q_{1y}(t_0)=q_{2y}(t_0), \qquad \dot{q}_{1y}(t_0)= \dot{q}_{2y}(t_0).  
 \end{equation}
 
 If $Z_2(t_0)$ is also on the $x$-axis and tangent to it, it implies that 
\[q_{3y}(t_0)=0, \qquad \dot{q}_{3y}(t_0)= 0. \]
 Note that the center of mass is fixed at $0$, it follows that 
 \begin{equation}\label{q1q2q3yall0}
 q_{1y}(t_0)=q_{2y}(t_0)=q_{3y}(t_0)=0, \qquad \dot{q}_{1y}(t_0)= \dot{q}_{2y}(t_0)=\dot{q}_{3y}(t_0)= 0.
 \end{equation}
The Newtonian equations and \eqref{q1q2q3yall0} imply that
\begin{equation}\label{qidouble0}
\ddot{q}_{1y}(t_0)= \ddot{q}_{2y}(t_0)=\ddot{q}_{3y}(t_0)= 0.
\end{equation}
Since the set \[\{(q_1, \, q_2, \, q_3) \,|\,  q_{1y}=q_{2y}=q_{3y}=0, \, \, \, \dot{q}_{1y}= \dot{q}_{2y}=\dot{q}_{3y}= 0\}\] is invariant, it imply that 
\[  q_{1y}(t)=q_{2y}(t)=q_{3y}(t)=0, \qquad \forall \, t \in (0, 1]. \]
It follows that both $Z_1(t)$ and $Z_2(t)$ stay on the $x$-axis for all $t \in (0, 1]$. 

If $Z_2(t_0)$ is tangent to the $y$-axis, we have 
\begin{equation}\label{q3xequal0}
q_{3x}(t_0)=0, \quad q_{1x}(t_0)=-q_{2x}(t_0), \quad \dot{q}_{3x}(t_0)= 0, \quad \dot{q}_{1x}(t_0)=-\dot{q}_{2x}(t_0). 
\end{equation}
Note that in \eqref{q1ya2yt0equal},
\[ q_{1y}(t_0)=q_{2y}(t_0), \qquad \dot{q}_{1y}(t_0)= \dot{q}_{2y}(t_0).  \]
By the Newtonian equations, \eqref{q1ya2yt0equal} and \eqref{q3xequal0} imply that
\[ \ddot{q}_{3x}(t_0)= 0, \qquad     \ddot{q}_{1y}(t_0)= \ddot{q}_{2y}(t_0). \]
Note that the set 
\[\{(q_1, \, q_2, \, q_3) \,|\,  q_{3x}=0,\,  \, q_{1x}=-q_{2x},  \, \,  q_{1y}=q_{2y}, \,\, \dot{q}_{3x}= 0, \, \, \dot{q}_{1x}=-\dot{q}_{2x}, \,  \, \dot{q}_{1y}= \dot{q}_{2y}\}\]
is invariant, it follows hat
\[  q_{3x}(t)=0,  \qquad q_{1y}(t)=q_{2y}(t), \quad \forall \, t \in (0, 1]. \]
It implies that $Z_1$ stays on the $x$-axis and $Z_2$ stays on the $y$-axis for all $t \in (0, 1]$. The proof is complete.
\end{proof}

\begin{proposition}\label{z1t1t2equivonaxes}
If there is a subinterval $[t_1, t_2] \subset (0, 1]$ such that $Z_1(t)$ stays on one of the coordinate axes, then both $Z_1(t)$ and $Z_2(t)$ must stay on the axes respectively for all $t \in (0, 1]$.
\end{proposition} 
\begin{proof}
In fact, without loss of generality, we can assume $Z_1(t)$ stays on the $x$-axis in $[t_1, t_2]$. That is, 
\[ q_{1y}(t)=q_{2y}(t), \quad \dot{q}_{1y}(t)= \dot{q}_{2y}(t), \quad  \ddot{q}_{1y}(t)= \ddot{q}_{2y}(t),  \quad \forall \, t \in [t_1, t_2]. \] 
By the Newtonian equations of $q_{1y}$ and $q_{2y}$, we have
\[ \frac{q_{3y}(t)-q_{1y}(t)}{|q_3(t)-q_1(t)|^3}=  \frac{q_{3y}(t)-q_{2y}(t)}{|q_3(t)-q_2(t)|^3}. \]
It follows that 
\[q_{3y}(t)=0, \quad \text{or} \quad |q_{3}(t)-q_{1}(t)|=|q_{3}(t)-q_{2}(t)|, \quad \forall \, t \in [t_1, t_2].   \]
Note that $q_{1y}(t)=q_{2y}(t)$, it implies that
\[q_{3y}(t)=0, \quad \text{or} \quad q_{3x}(t)=0 \, \, \, \text{and} \, \, \, q_{1x}(t)=-q_{2x}(t), \quad \text{or} \quad   q_{1x}(t)=q_{2x}(t), \quad \forall \, t \in [t_1, t_2].   \]
If $q_{1x}(t^{*})=q_{2x}(t^{*})$ for some $t^{*} \in [t_1, t_2]\subset (0,1]$, it implies that $Z_1(t(t^{*}))=0$, which is a collision. Contradiction! So for all $t \in [t_1, t_2]$, we have either $q_{3y}(t)=0$ or $q_{3x}(t)=0$. Note that $q_3$ is smooth in $[t_1, t_2]$, then we have either $q_{3y}\equiv 0$ or $q_{3x}\equiv 0$ for all $t \in [t_1, t_2]$. It implies that $Z_2(t)$ is on one of the axes and tangent to it for all $t \in [t_1, t_2]$. By Proposition \ref{z1z2sametimetangent}, it implies that both $Z_1(t)$ and $Z_2(t)$ must stay on the axes for all $t \in (0, 1]$. The proof is complete.
\end{proof}
Similarly, if $Z_2(t)$ stays on the axes in a closed subinterval of $(0,1]$, then both $Z_1(t)$ and $Z_2(t)$ stay on the axes respectively for all $t \in (0, 1]$.

\begin{lemma}\label{nocrossinginbetween01}
We assume that the path $(Z_1, \, Z_2) \in H_0^1([0,1],\, \mathbb{R}^4)$ minimizes the action $\mathcal{A}$ in \eqref{minPact}. Let $Z_1(0)=0$ and $Z_2(0) \neq 0$ be on the positive $x$-axis, and $\mathcal{A}(Z_1,\,  Z_2) = \mathcal{A}( \widetilde{Z}_1, \,  \widetilde{Z}_2)$. 

If either $Z_1=Z_1(t)$ or $Z_2=Z_2(t)$ does not stay on the axes for all $t \in [0,1]$, then $Z_i \, (i=1,2)$ can not cross the axes for $t \in (0,1)$. 
\end{lemma}
\begin{proof}
Note that $\mathcal{A}(Z_1,\,  Z_2) = \mathcal{A}( \widetilde{Z}_1, \,  \widetilde{Z}_2).$ By Lemma \ref{comaprevalueofU},  $Z_1(t)$ and $Z_2(t)$ must be in two adjacent closed quadrants for all $t \in (0,1]$. Furthermore, since both $(Z_1(t), \, Z_2(t))$ and $(\widetilde{Z}_1(t),  \, \widetilde{Z}_2(t))$ are smooth in $(0,1]$, it follows that whenever the path $Z_i \, (i=1,2)$ crosses one of the axes in $(0,1)$, the crossing must be non-transversal. 

By assumption, $Z_1(t) \not=0$ for all $t \in (0,1]$ and $Z_2(0) \not=0$. We first show $Z_2(t) \not=0$ for all $t \in (0,1)$ by contradiction. If not, we assume that $Z_2(t_{20})=0$ for some $t_{20} \in (0,1)$. By the smoothness of both $(Z_1, \, Z_2)$ and $( \widetilde{Z}_1, \,  \widetilde{Z}_2)$, it follows that the velocity $\dot{Z}_2(t_{20})=0$. Since $ \left\{ (q_1, \, q_2, \, q_3) \, |\, q_3=0, \, \dot{q}_3=0 \right\}$ is an invariant set, it implies that $Z_2 \equiv 0$ for all $t \in (0, 1]$. By continuity, $Z_2(0)=0$. Contradict to the assumption! Hence, $Z_i=Z_i(t) \, (i=1,2)$ can not reach the origin when $ t \in (0,1)$.

If $Z_1$ or $Z_2$ crosses the axes at $t=t_{0} \in (0,1)$, we show that both $Z_1(t_0)$ and $Z_2(t_0)$ must be tangent to the axes. In fact, without loss of generality, we assume $Z_1$ has a crossing with the axes. We claim that $Z_2(t_0)$ must be on the axes. Actually, if not, then $Z_2(t_0)$ is in some quadrant $\mathsf{Q}_i \, (i=1,2,3,4)$. By continuity, there exists small enough $\epsilon_0>0$, such that $Z_2$ is inside the same quadrant for all $t \in [t_{0}- \epsilon_0, t_{0}+ \epsilon_0]$. However, $Z_1$ is in two adjacent quadrants in the small interval $[t_{0}- \epsilon_0, t_{0}+ \epsilon_0]$. Since $Z_1(t)$ and $Z_2(t)$ must be in two adjacent closed quadrants for all $t \in (0,1]$, it follows that $Z_1(t)$ stays on the axes in one of the two intervals: $[t_{0}- \epsilon_0, t_{0}]$ and $[t_{0}, t_{0}+ \epsilon_0]$. By Proposition \ref{z1t1t2equivonaxes}, it implies that both $Z_1(t)$ and $Z_2(t)$ stay on the axes for all $t \in [0,1]$. Contradiction to the assumption!

Hence, $Z_2(t_0)$ is on the axes. If it is a crossing with the axes, by the smoothness of $Z_2$ and $\widetilde{Z}_2$, it's non-transversal, which implies that $Z_2(t_0)$ is tangent to the axes. If it only touches the axes, it is clear that it is tangent to it. It follows that both $Z_1(t_0)$ and $Z_2(t_0)$ are tangent to the axes. 

However, Proposition \ref{z1z2sametimetangent} implies that $Z_1(t)$ and $Z_2(t)$ must be on the axes for all $t \in [0,1]$. Contradiction to the assumption! Therefore, $Z_i \, (i=1,2)$ can not cross the axes for any $t \in (0,1)$. The proof is complete.
\end{proof}

\begin{theorem}\label{geometrictheorem}
We assume that the path $(Z_1, \, Z_2) \in H_0^1([0,1],\, \mathbb{R}^4)$ minimizes the action $\mathcal{A}$ in \eqref{minPact}. Let $Z_1(0)=0$ and $Z_2(0) \neq 0$ be on the positive $x$-axis, while $Z_1(1) \in \overline{\mathsf{Q}_1}$ and $Z_2(1)\in \overline{\mathsf{Q}_4}$. 

If $Z_1(1) \in \mathsf{Q}_1$ or $Z_2(1)\in \mathsf{Q}_4$, then the path $(Z_1,\, Z_2)$ must satisfy: $Z_1(t)\in \overline{\mathsf{Q}_1}$ and $Z_2(t)\in \overline{\mathsf{Q}_4}$ for all $ t\in[0,1]$. That is, $Z_i(t)= \widetilde{Z}_i(t)$ for all $t \in [0,1]$.

If both $Z_1(1)$ and $Z_2(1)$ are on the positive $x$-axis, then either $Z_1(t) \in \overline{\mathsf{Q}_1}, \, \, Z_2(t)\in \overline{\mathsf{Q}_4}$ for all $t \in [0,1]$, or $Z_1(t) \in \overline{\mathsf{Q}_4}, \, \, Z_2(t)\in \overline{\mathsf{Q}_1}$ for all $t \in [0,1]$.
\end{theorem}

\begin{proof}
By assumption, $Z_1(1) \in \overline{\mathsf{Q}_1}$ and $Z_2(1)\in \overline{\mathsf{Q}_4}$. It implies that $\mathcal{A}(Z_1,\,  Z_2) \leq \mathcal{A}( \widetilde{Z}_1, \,  \widetilde{Z}_2)$. By Lemma \ref{comaprevalueofU}, $\mathcal{A}(Z_1,\,  Z_2) \geq \mathcal{A}( \widetilde{Z}_1, \,  \widetilde{Z}_2)$. Hence, $\mathcal{A}(Z_1,\,  Z_2) = \mathcal{A}( \widetilde{Z}_1, \,  \widetilde{Z}_2)$.

If $Z_1(1) \in \mathsf{Q}_1$ or $Z_2(1)\in \mathsf{Q}_4$, By Lemma \ref{nocrossinginbetween01}, it follows that $Z_i \, (i=1,2)$ can not cross the axes in $(0,1)$. If $Z_1(1) \in \mathsf{Q}_1$, it implies that $Z_1(t) \in \mathsf{Q}_1$ for all $t \in [0,1]$. Note that $Z_1(t)$ and $Z_2(t)$ must be in two adjacent closed quadrants for all $t \in (0,1]$. Since $Z_2(0)$ is on the positive $x$-axis, it implies that $Z_2(t) \in \mathsf{Q}_4$ for all $t \in [0,1]$. Similarly, if $Z_2(1) \in \mathsf{Q}_4$, we still have $Z_1(t)\in \overline{\mathsf{Q}_1}$ and $Z_2(t)\in \overline{\mathsf{Q}_4}$ for all $ t\in[0,1]$. 

Next we proof the case when both $Z_1(1)$ and $Z_2(1)$ are on the positive $x$-axis.  If the path $(Z_1, \, Z_2)$ does not stay on the axes for all $t \in [0,1]$, by Lemma \ref{nocrossinginbetween01}, it follows that $Z_i \, (i=1,2)$ can not cross the axes in $(0,1)$. Then it can be either  $Z_1(t) \in \overline{\mathsf{Q}_1}, \, \, Z_2(t)\in \overline{\mathsf{Q}_4}$ for all $t \in [0,1]$, or $Z_1(t) \in \overline{\mathsf{Q}_4}, \, \, Z_2(t)\in \overline{\mathsf{Q}_1}$ for all $t \in [0,1]$. 

If the path $(Z_1, \, Z_2)=(Z_1(t), \, Z_2(t))$ is on the axes for all $t \in [0,1]$, they can only be on the $x$-axis since both $Z_1(1)= (b_2, \, -3 b_1)$ and $Z_2(1)=(\frac{3}{2}b_2,  \, \frac{3}{2}b_1)$ are away from the $y$-axis. By the argument in Lemma \ref{nocrossinginbetween01}, $Z_i(t) \not=0$ for all $t \in (0,1)$. It implies that $Z_i \in \overline{\mathsf{Q}_1}   \cap \overline{\mathsf{Q}_4} \, (i=1,2)$. Then it is clear that either $Z_1(t) \in \overline{\mathsf{Q}_1}, \, \, Z_2(t)\in \overline{\mathsf{Q}_4}$ for all $t \in [0,1]$, or $Z_1(t) \in \overline{\mathsf{Q}_4}, \, \, Z_2(t)\in \overline{\mathsf{Q}_1}$ for all $t \in [0,1]$ holds.  

The proof is complete.
\end{proof}

\begin{theorem}\label{mincolschubart}
If the minimizer $\mathcal{P}$ has a binary collision $q_1(0)=q_2(0)$, it must be a part of the Schubart orbit.
\end{theorem}
\begin{proof}
According to \cite{Ven1}, a quarter of the Schubart orbit is a minimizer on the $x$-axis from a binary collision between $q_1$ and $q_2$ to an Euler configuration with $q_1(1)=0$. Corollary \ref{180degreeejection} implies that if the minimizer has a binary collision $q_1(0)=q_2(0)$, then at $t=0$, $\vec{c}_{12}$ must be $(1,0)$. We then apply Theorem \ref{geometrictheorem} to prove that the path $\mathcal{P}$, i.e. the minimizer $(Z_1, \, Z_2)$, must stay on the $x$-axis.

To give the details of the proof, without loss of generality, we can assume that at $t=1$, $q_1(1)$ stays on the nonnegative part of the $y$-axis. Then in the minimizer $\mathcal{P}$, $Q_{e_1}$ can be one of the following two cases in Fig. \ref{dengyaogouxing}.

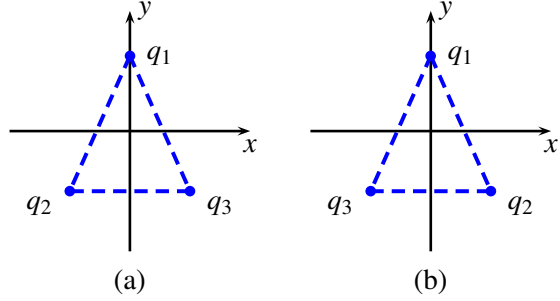
\begin{figure}[ht]
    \begin{center}
    \psset{xunit=2cm,yunit=2cm}
\begin{pspicture}(-2,-1)(2,1)
\psline[linewidth=1pt]{->}(-1.8, 0)(-0.2, 0)
\psline[linewidth=1pt]{->}(-1, -0.8)(-1, 0.8)
\psline[linewidth=1pt]{->}(0.2, 0)(1.8, 0)
\psline[linewidth=1pt]{->}(1, -0.8)(1, 0.8)
\rput(-0.2, -0.1){$x$}
\rput(1.8, -0.1){$x$}
\rput(-0.9, 0.8){$y$}
\rput(1.1, 0.8){$y$}
\psdots[dotsize=4pt, linecolor=blue](-1, 0.5)(1,0.5)(-1.4,  -0.4)(-0.6, -0.4)(1.4,  -0.4)(0.6, -0.4)
\psline[linewidth=1.5pt, linestyle=dashed,  linecolor=blue](-1, 0.5)(-1.4,  -0.4)(-0.6, -0.4)(-1, 0.5)
\psline[linewidth=1.5pt, linestyle=dashed,  linecolor=blue](1, 0.5)(1.4,  -0.4)(0.6, -0.4)(1, 0.5)
\rput(-0.8, 0.5){$q_1$}
\rput(1.2, 0.5){$q_1$}
\rput(-1.6, -0.5){$q_2$}
\rput(-0.4, -0.5){$q_3$}
\rput(1.6, -0.5){$q_2$}
\rput(0.4, -0.5){$q_3$}
\rput(-1, -1){(a)}
\rput(1, -1){(b)}
  \end{pspicture}   
   \end{center}
 \caption{ \label{dengyaogouxing}The isosceles configuration $Q_{e_1}$ of the minimizer $\mathcal{P}$ at $t=1$.}
\end{figure}

\textbf{Claim 1:} At $t=1$, $Q_{e_1}$ in $\mathcal{P}$ must be like Fig. \ref{dengyaogouxing} (a). 

In fact, by the definition of the boundary set $Q_{S_4}$ at $t=0$, we have 
\begin{equation}\label{initialposition}
Z_1(0)=(0,0),\ \ \ \ \ Z_2(0)=(3a_2, \, 0),
\end{equation}
where $a_2>0$. 

If the configuration at $t=1$ is like Fig. \ref{dengyaogouxing} (b), then $Z_1(1) \in \overline{\mathsf{Q}_2}$ and $Z_2(1)\in \overline{\mathsf{Q}_3}$. Let $\widetilde{Z}_1=(|Z_{1x}|, \, |Z_{1y}|)$ and $\widetilde{Z}_2= (|Z_{2x}|,\,  -|Z_{2y}|)$. It is clear that $\widetilde{Z}_1(t) \in \overline{\mathsf{Q}_1}$ and $\widetilde{Z}_2(t) \in \overline{\mathsf{Q}_4}$ for all $t \in [0,1]$, and the two paths $(Z_1,\,  Z_2)$ and $(\widetilde{Z}_1, \, \widetilde{Z}_2)$ must be different. 

 Note that $(\widetilde{Z}_1, \, \widetilde{Z}_2)$ is a path in $\{ q \in H^1([0,1], \chi) \, | \, q(0) \in Q_{S_4}, \, q(1) \in Q_{E_1} \}$. By the same argument as in Lemma \ref{comaprevalueofU}, we have $\mathcal{A}(Z_1, \, Z_2) \geq \mathcal{A}(\widetilde{Z}_1, \, \widetilde{Z}_2)$. Next, we show that 
 \[\mathcal{A}(Z_1, \, Z_2)>\mathcal{A}(\widetilde{Z}_1, \, \widetilde{Z}_2).\] 
If not, it implies that $\mathcal{A}(Z_1,\, Z_2)=\mathcal{A}(\widetilde{Z}_1, \, \widetilde{Z}_2)$. It follows that both $(Z_1,\, Z_2)$ and $(\widetilde{Z}_1,\, \widetilde{Z}_2)$ are analytic in $(0,1]$. Note that $Z_2(0) \not=0$ is on the positive $x$-axis, while $Z_2(1) \in \overline{\mathsf{Q}_3}$ and $Z_2(1) \not=0$. It implies that there must exists some $t_0 \in (0,1)$, such that $Z_2$ crosses the $y$-axis at $t=t_0$ non-transversally. A similar argument as in Lemma \ref{nocrossinginbetween01} implies that both $Z_1(t_0)$ and $Z_2(t_0)$ must be tangent to the axes. By Proposition \ref{z1z2sametimetangent}, it follows that the path $(Z_1, \, Z_2)=(Z_1(t), \, Z_2(t))$ must stay on the axes for all $t \in [0,1]$. Since $Z_i(1) \, (i=1,2)$ can not be on the $y$-axis, it follows that the path $(Z_1, Z_2)$ is always on the $x$-axis and $Z_2(t)$ must cross the origin in $(0,1)$. However, by the proof in Lemma \ref{nocrossinginbetween01}, $Z_2(t)$ can not cross the origin for $t \in (0,1)$. Contradiction! Hence,  
\begin{equation*}
  \mathcal{A}(Z_1,Z_2)> \mathcal{A}(\widetilde{Z}_1,\, \widetilde{Z}_2). 
\end{equation*}
Thus, the configuration $Q_{e_1}$ of the minimizer $\mathcal{P}$ at $t=1$ must be like Fig. \ref{dengyaogouxing} (a). Claim 1 is proved.

Therefore,  $Z_1(1) \in \overline{\mathsf{Q}_1}$ and $Z_2(1)\in \overline{\mathsf{Q}_4}$. Note that $Z_1(1) \not= 0$ and $Z_2(1) \not= 0$ can not be on the $y$-axis. By Theorem \ref{geometrictheorem}, it implies that $Z_1(t)\in \overline{\mathsf{Q}_1}$ and $Z_2(t)\in \overline{\mathsf{Q}_4}$ for all $ t\in[0,1]$ if either $Z_1(1)$ or $Z_2(1)$ is away from the $x$-axis. And if both $Z_1(1)$ and $Z_2(1)$ are on the positive $x$-axis, then either $Z_1(t)$ stays in the first quadrant and $Z_2(t)$ stays in the fourth quadrant for all $t \in [0,1]$, or $Z_1(t)$ stays in the fourth quadrant and $Z_2(t)$ stays in the first quadrant for all $t \in [0,1]$.

On the other hand, note that $\vec{c}_{12}=(1,0)$. At $t=0$, by \cite{Ven1}, the velocities of body 3 satisfies $\dot{q}_{3x}(0)=0$. By Lemma \ref{asympvec}, the velocities in the $y$-axis satisfy
\begin{equation}\label{yvelocity}
  \dot{q}_{1y}(0)=\dot{q}_{2y}(0)= -\frac{1}{2}\dot{q}_{3y}(0).
\end{equation}

If $\dot{q}_{1y}(0)=\dot{q}_{2y}(0)=0$, by Lemma 6.2 in \cite{Yu2}, the motion must be collinear. In this case, the minimizer $\mathcal{P}$ coincides with the Schubart orbit. 

Note that $q_{1y}(0)= q_{2y}(0)=0$ and $\dot{q}_{1y}(0)=\dot{q}_{2y}(0)$. When $\epsilon>0$ small enough, we consider the following identity:
\begin{equation}\label{doubledotqident}
q_{1y}(\epsilon)- q_{2y}(\epsilon)= \int_0^{\epsilon} \int_0^t \left[\ddot{q}_{1y}- \ddot{q}_{2y} \right] \, ds \,dt.
\end{equation}
If $\dot{q}_{1y}(0)=\dot{q}_{2y}(0)>0$, then for small enough $\epsilon_1>0$, $Z_2(\epsilon_1) \in \mathsf{Q}_4$. By Theorem \ref{geometrictheorem},  it follows that $Z_2(t) \in \overline{\mathsf{Q}_4}$ and $Z_1(t) \in \overline{\mathsf{Q}_1}$. By Proposition \ref{z1t1t2equivonaxes}, we can choose small enough $\epsilon>0$, such that $Z_1(\epsilon) \in \mathsf{Q}_1$. That is $Z_{1y}(\epsilon)= q_{1y}(\epsilon)- q_{2y}(\epsilon)>0$.   

However, the Newtonian equations imply that    
\[\ddot{q}_{1y}- \ddot{q}_{2y} = \frac{2(q_{2y}-q_{1y})}{|q_1-q_2|^3} + \frac{q_{3y}-q_{1y}}{|q_3-q_1|^3}-\frac{q_{3y}-q_{2y}}{|q_2-q_3|^3}.  \]
Since $Z_1(t) \in \overline{\mathsf{Q}_1}$ and $Z_2(t) \in \overline{\mathsf{Q}_4}$ for all $t \in [0, \epsilon]$ with $\epsilon>0$ small enough, it follows that for all $t \in [0, \epsilon]$, 
\[\frac{2(q_{2y}(t)-q_{1y}(t))}{|q_1(t)-q_2(t)|^3}\leq 0, \quad |q_3(t)-q_1(t)|\leq |q_3(t)-q_2(t)|, \quad  q_{3y}(t)-q_{1y}(t) \leq q_{3y}(t)-q_{2y}(t)\leq 0.\]
Hence, 
\[\frac{2(q_{1y}(t)-q_{2y}(t))}{|q_1(t)-q_2(t)|^3}\leq 0, \qquad \frac{q_{3y}(t)-q_{1y}(t)}{|q_3(t)-q_1(t)|^3}-\frac{q_{3y}(t)-q_{2y}(t)}{|q_2(t)-q_3(t)|^3}\leq 0, \quad \forall \, t \in [0, \epsilon]. \]
It implies that 
\[q_{1y}(\epsilon)- q_{2y}(\epsilon)= \int_0^{\epsilon} \int_0^t \left[\ddot{q}_{1y}- \ddot{q}_{2y} \right] \, ds \,dt \leq 0. \]
Contradict to $q_{1y}(\epsilon)- q_{2y}(\epsilon)>0$!

Similarly, $\dot{q}_{1y}(0)=\dot{q}_{2y}(0)<0$ can NOT hold. 

Therefore, the minimizer $\mathcal{P}$ satisfies $\dot{q}_{1y}(0)=\dot{q}_{2y}(0)=0$, and it must coincide with the Schubart orbit. The proof is complete.
\end{proof}

\section{First variation and extension of the minimizing path}
This section shows that  if the minimizer $\mathcal{P}$ in the minimizing problem \eqref{minimizerhenon} has no collision, it can be extended to the Broucke-H\'{e}non orbit in Fig.~\ref{Henonorbit}. Recall that
\begin{equation*}
 Q_{s_1} = \begin{bmatrix}
-2a_1-a_2   &  0   \\
a_1-a_2  &  0     \\
a_1+2a_2   & 0   
\end{bmatrix}, \qquad \,\,
  Q_{e_1} =  \begin{bmatrix}
0   &  -2b_1    \\
-b_2 &  b_1 \\
b_2   &   b_1
\end{bmatrix},  
\end{equation*}
where $q=(q_1^T, q_2^T, q_3^T)^T$ and $q_i  =(q_{ix}, q_{iy})\ (i=1,2,3)$ are row vectors in $\mathbb{R}^2$. The set $\mathcal{S}$ is 
\begin{eqnarray*}
\mathcal{S}= \bigg\{ (a_1, \, a_2,  \,  b_1, \, b_2) \, \big| \,  a_1 \geq 0, \, \, a_2 \geq 0, \, b_1\in \mathbb{R}, \,  b_2 \in \mathbb{R} \bigg\}. 
\end{eqnarray*}
According to Theorem \ref{Thm1.1}, in the minimizing path $\mathcal{P}$, the parameters $a_{10}, \, a_{20}, \, b_{10}, \, b_{20} \in \mathcal{S}$ are all finite. The boundary of $\mathcal{S}$: $a_{1}=0$ or $a_{2}=0$ corresponds to collisions. The assumption that path $\mathcal{P}$ is collision-free implies that  $a_{10}, a_{20}$ are away from the boundary of $\mathcal{S}$. Hence,  the first variation formula can be applied to the minimizing path $\mathcal{P}$.

\begin{lemma}\label{velocitiesofallbodies}
Let $q^{*}(t) \, (t\in [0,1])$ be the positions of the action minimizer $\mathcal{P}$. Let $q^{*}_i(t)$ be the position of the $i$-th body and $\dot{q^{*}_i}(t)= \left(\dot{q}^{*}_{ix}(t), \dot{q}^{*}_{iy}(t) \right) $ be its corresponding velocity.  Then
\begin{equation}\label{velocityat0}
\dot{q}^{*}_{1x}(0)=\dot{q}^{*}_{2x}(0)=\dot{q}^{*}_{3x}(0)=0,
\end{equation}
\begin{equation}\label{velocityat10}
\dot{q}^{*}_{1y}(1) =0, \qquad \dot{q}^{*}_{2y}(1)=- \dot{q}^{*}_{3y}(1),  
\end{equation}
\begin{equation}\label{velocityat11}
\dot{q}^{*}_{2x}(1) =  \dot{q}^{*}_{3x}(1).
\end{equation}
\end{lemma}

\begin{proof}
The proof of \eqref{velocityat0}, \eqref{velocityat10} and \eqref{velocityat11} are essentially the same. Here we only show \eqref{velocityat0} in detail. Let $\xi(0)$ be an admissible variation, which means that $q^* + \delta \xi \in \{ q \in H^1([0,1], \chi) | q(0) \in Q_{S_1}, \, q(1) \in Q_{E_1} \}$ for small enough  $\delta$, where $Q_{S_1}= \{Q_{s_1} \, | \, a_1 \geq 0, a_2 \geq 0\}$ and $Q_{E_1}= \{Q_{e_1} \, | \, b_1\in \mathbb{R}, b_2 \in \mathbb{R}\}$. The first variation $\delta_{\xi} \mathcal{A}(q^*)$ satisfies
\begin{eqnarray*}
& & \delta_{\xi} \mathcal{A}(q^*)\\
&=& \lim_{\delta \to 0} \frac{\mathcal{A}(q^* + \delta \xi) -   \mathcal{A}(q^* ) }{\delta}\\
&=&\lim_{\delta \to 0} \int_{0}^1 \, \sum_{i=1}^3 \frac{ m_i |\dot{q}^* + \delta \dot{\xi}|^2- m_i |\dot{q}^*|^2 }{2\delta} + \frac{U(q^* + \delta \xi )- U(q^*)}{\delta}  \,  dt\\
&=& \int_{0}^1 \, \sum_{i=1}^3 m_i <\dot{q}^*_i, \dot{\xi}_i> + \sum_{i=1}^3 <\frac{\partial U}{ \partial q^*_i} , \xi_i> \, dt \\
&=& \sum_{i=1}^3 m_i < \dot{q}^*_i,  \xi_i > \bigg|_0^1 + \int_{0}^1 \,  \sum_{i=1}^3 < -m_i  \ddot{q}^{*}_i +\frac{\partial U}{\partial q^*_i},  \xi_i> \, dt\\
&=&\sum_{i=1}^3 m_i< \dot{q}^*_i,  \xi_i > \bigg|_0^1 .
\end{eqnarray*}
Because the first variation vanishes for any $\xi$ and the minimizing path $q^*$,  $\sum_{i=1}^3 m_i< \dot{q}^*_i,  \xi_i > \bigg|_0^1=0$. In particular, $\xi(0)$ can be $ \begin{pmatrix}
-2   &  0 \\
1   &  0\\
1  &  0
\end{pmatrix}$ and $\xi(1)$ can be $0$. It follows that
\begin{equation}\label{velocityalltheta}
2\dot{q}^{*}_{1x}(0)-\dot{q}^{*}_{2x}(0)-\dot{q}^{*}_{3x}(0)=0.
\end{equation}
We can set $\xi(0)= \begin{pmatrix}
-1   &  0 \\
-1   &  0\\
2  &  0
\end{pmatrix}$ and $\xi(1)=0$, it follows that
\begin{equation}\label{thetanot0pi2}
-\dot{q}^{*}_{1x}(0)- \dot{q}^{*}_{2x}(0)+ 2\dot{q}^{*}_{3x}(0)=0.
\end{equation}
Note that the total linear momentum $ \dot{q}^{*}_{1x}(0)+ \dot{q}^{*}_{2x}(0)+ \dot{q}^{*}_{3x}(0) =0$. It follows that  $ \dot{q}^{*}_{1x}(0)= \dot{q}^{*}_{2x}(0)= \dot{q}^{*}_{3x}(0) =0$.

The other three identities \eqref{velocityat10} and \eqref{velocityat11} can be shown by similar arguments. The proof is complete.
\end{proof}

With the help of Lemma \ref{velocitiesofallbodies}, we can show that $q^*(t) \, (t \in [0,1])$ can be extended to a periodic solution.
\begin{lemma}\label{henonorbit1}
$q^*(t) \, (t \in [0,1])$ can be extended to a periodic orbit with period $T=4$.
\end{lemma}

\begin{proof}
At $t=0$, three bodies line up on the $x$-axis in an order $q_1(0) < q_2(0) <q_3(0)$. At $t=1$, they form an isosceles triangle with symmetry axis on the $y$-axis. By Lemma \ref{velocitiesofallbodies}, $\dot{q}^{*}_{1x}(0)= \dot{q}^{*}_{2x}(0)= \dot{q}^{*}_{3x}(0) =0$ and $\dot{q}^{*}_{1y}(1)=0, \, \dot{q}^{*}_{2x}(1)= \dot{q}^{*}_{3x}(1)$. Let $q^*(t) = \begin{bmatrix}
q^*_1(t) \\
  q^*_2(t)\\
   q^*_3(t)
   \end{bmatrix}= \begin{bmatrix}
q^*_{1x}(t)  &  q^*_{1y}(t)  \\
q^*_{2x}(t)  &  q^*_{2y}(t)\\
 q^*_{3x}(t)  &  q^*_{3y}(t)
   \end{bmatrix}$, where $t \in [0, 1]$. When $t \in (1,2]$, we define 
    \begin{equation}\label{qsol1to2}
    q^*(t) = \begin{bmatrix}
q^*_1(t) \\
  q^*_2(t)\\
   q^*_3(t)
   \end{bmatrix}=    \begin{bmatrix}
 -q^*_{1x}(2-t)  &  q^*_{1y}(2-t)  \\
-q^*_{3x}(2-t)  &  q^*_{3y}(2-t)\\
 -q^*_{2x}(2-t)  &  q^*_{2y}(2-t)
   \end{bmatrix}, \qquad t\in (1,2]. 
   \end{equation}
   When $t$ approaches 1, it is easy to check that  $\lim_{t\rightarrow 1^-}q^*(t)=\lim_{t\rightarrow 1^+}q^*(t)= q^*(1)$.  On the other hand, by applying Lemma  \ref{velocitiesofallbodies}, it follows that
  \[ \dot{q}^*(1)=\begin{bmatrix}
 \dot q^*_{1x}(1)  &  -\dot q^*_{1y}(1)  \\
\dot q^*_{3x}(1)  &  -\dot q^*_{3y}(1)\\
 \dot q^*_{2x}(1)  &  -\dot q^*_{2y}(1)
   \end{bmatrix}= \begin{bmatrix}
 \dot q^*_{1x}(1)  &  0\\
\dot q^*_{2x}(1)  &  \dot q^*_{2y}(1)\\
 \dot q^*_{3x}(1)  &  \dot q^*_{3y}(1)
   \end{bmatrix}. \]
Hence, at $t=1$, $q^*(t)  \, (t \in (1,2])$  and $q^*(t)  \, (t \in [0,1])$ are smoothly connected. By the uniqueness of solution of initial value problem in an ODE system, $q^*(t)$ can be extended to $[0,2]$ by \eqref{qsol1to2}. At $t=2$, 
\[q^*_1(2)= \left(2a_{10}+a_{20}, \, 0 \right), \qquad q^*_2(2)= \left(-a_{10}-2a_{20}, \, 0 \right),  \qquad q^*_3(2)= \left(-a_{10}+a_{20}, \, 0 \right). \]
It follows that, at $t=2$ the three bodies line up on the $x$-axis again in an order $q_2(2) < q_3(2) <q_1(2)$. By Lemma \ref{velocitiesofallbodies}, the velocities of the three bodies at $t=2$ are all vertical. Therefore, we can extend the path to $t \in (2, 4]$ as follows
 \begin{equation}\label{qsol2to4}  q^*(t) = \begin{bmatrix}
q^*_1(t) \\
  q^*_2(t)\\
   q^*_3(t)
   \end{bmatrix} =    \begin{bmatrix}
 q^*_{1x}(4-t)  &  -q^*_{1y}(4-t)  \\
q^*_{2x}(4-t)  &  -q^*_{2y}(4-t)\\
 q^*_{3x}(4-t)  &  -q^*_{3y}(4-t)
   \end{bmatrix}, \qquad t\in(2,4]. \end{equation}
It follows that at $t=4$, $ q^*_{i}(4)= q^*_i(0)\, (i=1,2,3)$ and the velocities satisfy  $\dot{q}^*_{i}(4)= \dot{q}^*_i(0)\, (i=1,2,3)$. Hence, $q^*(t)$ can be extended to a periodic solution by \eqref{qsol1to2} 
and \eqref{qsol2to4}, which has a period $4$. The proof is complete.
\end{proof}

At the end of this section, we show that if the minimizer $\mathcal{P}$ has no collision, then the periodic orbit generated by $\mathcal{P}$ has a $D_2$ symmetry. 
\begin{lemma}
For any $t\in\mathbb{R}$, 
\begin{equation}
q^*_i(t) = R_xq^*_i(-t),   \qquad (i = 1, 2, 3), 
\end{equation}
\begin{equation}
q^*_1(t+2) = R_xR_yq^*_1(t),   \quad q^*_2(t+2)=R_xR_yq^*_3(t),  \quad q^*_3(t+2)=R_xR_yq^*_2(t),
\end{equation}
where $R_x$ and $R_y$ is defined as follows
\[R_x= \begin{bmatrix}
1   &  0   \\
0  &  -1   \\ 
\end{bmatrix}, \hspace{1cm} R_y= \begin{bmatrix}
-1   &  0   \\
0  &  1   \\ 
\end{bmatrix}.\]
\end{lemma}
\begin{proof}
Actually, by the extension formula \eqref{qsol2to4}, it is clear that 
\[q^*_i(t) = R_xq^*_i(4-t) = R_xq^*_i(-t), \quad (i = 1, 2, 3). \]
By the extension formula \eqref{qsol1to2}, we have
\[q^*_1(2-t)= R_y  q^*_1(t), \quad   q^*_2(2-t)= R_y  q^*_3(t), \quad   q^*_3(2-t)= R_y  q^*_2(t). \]
It follows that 
\[q^*_1(2+t)= R_xR_y q^*_1(t), \quad   q^*_2(2-t)= R_xR_y   q^*_3(t), \quad   q^*_3(2-t)= R_xR_y   q^*_2(t). \]
The proof is complete.
\end{proof}

\begin{remark}
According to our analysis, the minimizer $\mathcal{P}$ must coincide with either the Schubart orbit or the Broucke-H\'{e}non orbit. Let the masses be $m_1=m_2=m_3=1$. Numerically, we can calculate the action value of $\mathcal{P}$  in each case. If $\mathcal{P}$ is one part of the Schubart orbit, its action value $\mathcal{A}_{10}$ is
\[\mathcal{A}_{10}  \approx 3.43.  \]
The positions and velocities of the Schubart orbit with minimum period $T=4$ at $t=1$ are
\begin{align*}
 q_1(1)&=(0, \,    0),   & \dot{q}_1(1)&=(-0.6328, \,   0); &\\
q_2(1)&=(-1.7141,  \,   0),   & \dot{q}_2(1)&=(0.3164, \,   0);& \\
q_3(1)&=(1.7141, \,   0),    & \dot{q}_3(1)&=(0.3164, \,   0). &
\end{align*}
If $\mathcal{P}$ is one part of the Broucke-H\'{e}non orbit, its action value $\mathcal{A}_{20}$ is 
\[\mathcal{A}_{20} \approx  3.46. \]
The initial condition of the Broucke-H\'{e}non orbit with minimum period $T=4$ is
\begin{align*}
 q_1(0)&=(-0.9031, \,  0),   &\dot{q}_1(0)&=(0,  \,   -2.4504); &\\
q_2(0)&=(-0.7321,  \,   0),   &\dot{q}_2(0)&=(0, \,    2.2283);& \\
q_3(0)&=(1.6352, \,   0),    &\dot{q}_3(0)&=(0, \,   0.2221). &
\end{align*}
Therefore, numerical evidence implies that the action minimizer $\mathcal{P}$ of \eqref{thispapergoal} coincides with the Schubart orbit. 
\end{remark}

%


\begin{thebibliography}{00}

 \bibitem{CM} A. Chenciner, R. Montgomery, A remarkable periodic solution of the three-body problem in the case of equal masses, {\em Ann. of Math.} {\bf 152} (2000) 881--901.

 \bibitem{CA2} A. Chenciner, Action minimizing solutions in the Newtonian n-body problem: from homology to symmetry, {\em Proceedings of the International Congress of Mathematicians (Beijing, 2002)}, Higher Ed. Press, Beijing, 279--294, 2002. 

\bibitem{BR} R. Broucke, On relative periodic solutions of the planar general three-body problem,  {\em Celest. Mech.} {\bf 12} (1975) 439--462.

\bibitem{Chen2} K. Chen, Existence and minimizing properties of retrograde orbits to the three-body problem with various choices of masses, {\em Ann. of Math.} {\bf 167} (2008) 325--348.

\bibitem{Chen3} K. Chen, Removing collision singularities from action minimizers for the N-body problem with free boundaries, {\em Arch. Rational Mech. Anal.} {\bf 181} (2006) 311--331.

\bibitem{Chen1}K. Chen, Y. Lin, On action-minimizing retrograde and prograde orbits of the three-body problem, {\em Commun. Math. Phys.} {\bf 291} (2009) 403--441.

\bibitem{Chen}  K. Chen, T. Ouyang, Z. Xia, Action-minimizing periodic and quasi-periodic solutions in the n-body problem, {\em Math. Res. Lett.} {\bf 19} (2012) 483--497.

\bibitem{Chen4} K. Chen, A minimizing property of hyperbolic Keplerian orbits, {\em J. Fixed Point Theory Appl.} {\bf 19} (2017) 281--287.

\bibitem{Fusco} G. Fusco, G. Gronchi, P. Negrini, Platonic polyhedra, topological constraints and periodic solutions of the classical N-body problem, {\em Invent. math.} {\bf 185} (2011) 283--332.

\bibitem{FT} D. Ferrario, S. Terracini, On the existence of collisionless equivariant minimizers for the classical n-body problem, {\em Invent. math.} 155 (2004) 305--362.

\bibitem{HM1} M. H\'{e}non, Families of periodic orbits in the planar three-body problem, {\em Celest. Mech.} {\bf 10} (1974) 375--388.

\bibitem{HM} M. H\'{e}non, A family of periodic solutions of the planar three-body problem, and their stability, {\em Celest. Mech.} {\bf 13} (1976) 267--285.

\bibitem{Gordon} W. B. Gordon, A minimizing property of Keplerian orbits, {\em Amer. J. Math.} {\bf 99} (1977) 961--971.

\bibitem{Long}  Y. Long, S. Zhang, Geometric characterizations for variational minimization solutions of the 3-body problem, {\em Acta Math. Sin. (Engl. Ser.)} {\bf 16} (2000) 579--592.

\bibitem{Ma1} C. Marchal, How the method of minimization of action avoids singularities, {\em Celest. Mech. Dyn. Astro.} {\bf 83} (2002) 325--353.

\bibitem{Moe} R. Moeckel, A topological existence proof for the Schubart orbits in the collinear three-body problem. {\em  Disc. Cont. Dyn. Syst. Ser. B} {\bf 10} (2008) 609--620.

\bibitem{OuXie} T. Ouyang, Z. Xie,  Star pentagon and many stable choreographic solutions of the Newtonian 4-body problem, {\em  Phys. D} {\bf 307} (2015) 61--76.





\bibitem{Sarri} D. Saari, The manifold structure for collision and for hyperbolic-parabolic orbits in the N-body problem, {\em J. Diff. Eqn.} {\bf 55} (1984) 300--329.

\bibitem{Simo} C. Sim{\'o}, E. Lacomba, Regularization of simulataneous binary collisions in the N-body problem, {\em J.Diff. Eqn. } {\bf 55} (1992) 241--259. 

\bibitem{Sch}  J. Schubart, Numerische aufsuchung periodischer l\"{o}sungen im dreik\"{o}rperproblem, {\em Astr. Nachr.} {\bf 283} (1956) 17--22.

\bibitem{SH} H.J. Sperling, On the real singularities of the $N$-body problem, {\em J. Reine Angew. Math.} {\bf 245} (1970) 15--40.

\bibitem{SK} K.F. Sundman, M\'{e}moire sur le probl\`{e}des trois corps. {\em Acta Math.} {\bf 36} (1913) 105--179.

\bibitem{Ven} A. Venturelli, Application de la minimisation de l'action au probl\`{e}me des N corps dans le plan et dans l'espace. {\em Thesis, Universit\'{e} de Paris 7}, 2002.

\bibitem{TV} S. Terracini, A. Venturelli, Symmetric trajectories for the $2N$-body problem with equal masses, {\em Arch. Rational Mech. Anal. } {\bf 184} (2007) 465--493.

\bibitem{Ven1} A. Venturelli, A variational proof of the existence of von Schubart's orbit, {\em Disc. Cont. Dyn. Syst. Ser. B} {\bf 10} (2008) 699--717.

\bibitem{Ven2} E. Mateus, A. Venturelli, C. Vidal, Quasiperiodic collision solutions in the spatial isosceles three-body problem with rotating axis of symmetry, {\em Arch. Ration. Mech. Anal.} {\bf 210} (2013) 165--176.

\bibitem{Yan} T. Ouyang, D. Yan, Simultaneous binary collisions in the equal-mass collinear four-body problem, {\em Electron. J. Differ. Eqns. } {\bf 80} (2015) 1--34. 

\bibitem{YAN2} D. Yan, T. Ouyang, New phenomena in the spatial isosceles three-body problem, {\em Int. J. Bifur. Chaos} {\bf 25} (2015) 1550116.

\bibitem{Yan3} D. Yan, A simple existence proof of Schubart periodic orbit with arbitrary masses, {\em Front. Math. China} {\bf 7} (2012) 145--160.

\bibitem{YAN4} D. Yan, T. Ouyang, Z. Xie, Classification of periodic orbits in the planar equal-mass four-body problem, {\em Discrete Contin. Dyn. Syst. 2015, Dynamical systems, differential equations and applications. 10th AIMS Conference. Suppl.} 1115--1124.

\bibitem{YAN5} S. Han, A. Huang, T. Ouyang, D. Yan, New periodic orbits in the planar equal-mass five-body problem, {\em Commun. Nonlinear Sci. Numer. Simul.} {\bf 48} (2017) 425--438.

\bibitem{YK} W. Kuang, D. Yan, Existence of prograde double-double orbits in the equal-mass four-body problem, {\em Preprint}, 2017.

\bibitem{Yu} G. Yu, Periodic solutions of the planar N-center problem with topological constraints, {\em Disc. Cont. Dyn. Syst.} {\bf 9} (2016) 5131--5162.

\bibitem{Yu2} G. Yu, Shape space figure-8 solution of three body problem with two equal masses, {\em Nonlinearity} {\bf 30} (2017) 2279--2307.







\end{thebibliography}
\end{document}